\documentclass[11pt,a4paper]{article}
\pdfoutput=1
\usepackage{ae,aecompl}
\usepackage{amssymb,amsmath,amsthm}
\usepackage[bookmarksnumbered=true,pdfstartview=FitH]{hyperref}
\usepackage[english]{babel}
\usepackage{fancyhdr}
\usepackage[top=3cm,right=2.5cm,left=2.5cm,bottom=2.3cm]{geometry}

\numberwithin{equation}{section}
\allowdisplaybreaks[4]

\newtheorem{thm}{Theorem}[section]
\newtheorem{lemma}[thm]{Lemma}
\newtheorem{prop}[thm]{Proposition}
\newtheorem{cor}[thm]{Corollary}
\newtheorem{df}[thm]{Definition}

\newtheorem{rem}[thm]{Remark}

\newcommand{\ot}{\otimes}

\newcommand{\A}{\mathcal{A}}
\newcommand{\B}{\mathcal{B}}
\newcommand{\K}{\mathcal{K}}
\newcommand{\HH}{\mathcal{H}}
\newcommand{\M}{\mathcal{M}}
\newcommand{\R}{\mathbb{R}}
\newcommand{\C}{\mathbb{C}}
\newcommand{\N}{\mathbb{N}}
\newcommand{\Z}{\mathbb{Z}}
\newcommand{\Q}{\mathbb{H}}

\newcommand{\D}{D\mkern-11.5mu/\,}
\newcommand{\inner}[1]{\left<#1\right>}
\newcommand{\qiso}{\smash[t]{\mathrm{QISO}}\rule{0pt}{8pt}^+_J}
\newcommand{\qisot}{\widetilde{\rule{0pt}{7pt}\smash[t]{\mathrm{QISO}}}\rule{0pt}{8pt}^+_J}

\newenvironment{mat}{\bigg(\!\!\begin{array}{cc}}{\end{array}\!\!\bigg)}

\newcommand{\mc}[1]{\mathcal{#1}}
\renewcommand{\bar}{\overline}
\newcommand{\tr}{\mathrm{Tr}}

\newcommand{\qaut}{\mathbf{C_{aut,\R}}}

\newcommand{\id}{\mathtt{id}}

\title{\vspace{-1cm}Quantum gauge symmetries in noncommutative geometry}

\date{}

\author{~\\
Jyotishman Bhowmick$^{\,1}$,
Francesco D'Andrea$^{\,2}$,
Biswarup Das$^{\,3}$,
Ludwik D{\k a}browski$^{\,4}$\\[20pt]
\textit{\small $^1$ University of Oslo, Moltke Moes vei 35, 0316 Oslo, Norway}\\[4pt]
\textit{\small $^2$ Universit\`a degli Studi di Napoli Federico II,
P.le Tecchio 80, I-80125 Naples, Italy}\\[4pt]
\textit{\small $^3$ Indian Statistical Institute, 203, B.T.~Road, Kolkata, India}\\[4pt]
\textit{\small $^4$ Scuola Internazionale Superiore di Studi Avanzati (SISSA), via Bonomea 265, I-34136 Trieste, Italy}}

\begin{document}

\maketitle

\begin{abstract}\noindent%
We discuss generalizations of the notion of i) the group of unitary elements
of a (real or complex) finite dimensional $C^*$-algebra, ii) gauge transformations
and iii) (real) automorphisms, in the framework of compact quantum group theory
and spectral triples.
The quantum analogue of these groups are defined as universal (initial) objects
in some natural categories. After proving the existence of the universal objects,
we discuss several examples that are of interest to physics, as they appear in
the noncommutative geometry approach to particle physics: in particular,
the $C^*$-algebras $M_n(\R)$, $M_n(\C)$ and $M_n(\Q)$, describing the finite
noncommutative space of the Einstein-Yang-Mills systems, and the algebras
$\A_F=\C\oplus\Q\oplus M_3(\C)$ and $\A^{\mathrm{ev}}=\Q\oplus\Q\oplus M_4(\C)$,
that appear in Chamseddine-Connes derivation of the Standard Model of
particle physics minimally coupled to gravity. As a byproduct, we identify
a ``free'' version of the symplectic group $Sp(n)$ (quaternionic unitary group).
\end{abstract}


%

\section{Introduction}

In the approach to particle physics from noncommutative geometry \cite{Con94,CC97}, the dynamics of a theory is obtained from the asymptotic expansion of the spectral action associated to an almost commutative spectral triple $(\A^\infty,\HH,D,J)$, i.e.~a product of the canonical spectral triple of a spin manifold and a finite-dimensional one (see e.g.~\cite{CM08} and references therein).
A fundamental role is played by the group $U(\A^\infty)$ of unitary elements of the algebra, whose adjoint representation $u\mapsto uJuJ^{-1}$ on $\HH$ gives the group
\begin{equation}\label{eq:one}
\mc{G}(\A^\infty,J)=\{uJuJ^{-1},u\in U(\A^\infty)\}
\end{equation}
of inner fluctuations of the real spectral triple \cite[Sec.~10.8]{CM08}, also called ``gauge group'' of the spectral triple, for its relation with the gauge group of physics \cite{walter}.
For example, in the Einstein-Yang-Mills system, the finite-dimensional spectral triple describing the internal noncommutative space is built from the algebra $\A_I=M_n(\C)$, with Hilbert space $\HH_I=M_n(\C)$ carrying the left regular representation and real structure $J_I$ given by the hermitian conjugation; in this case $U(\A_I)=U(n)$ and $\mc{G}(\A_I,J_I)$ is the classical gauge group $SU(n)$, modulo a finite group given by its
center. In the more elaborated example of (the Euclidean version of) the Standard Model of elementary particles minimally coupled to gravity, the algebra is $\A_F=\C\oplus\Q\oplus M_3(\C)$ and the group $\mc{G}(\A_I,J_I)$ is $U(1)\times SU(2)\times SU(3)$ modulo $\Z_6$.

\smallskip

More generally, consider a spectral triple based on an almost commutative algebra
$$
\A^\infty:=C^\infty(\mc{M})\otimes \A_I\simeq C^\infty(\mc{M}\to \A_I) \;,
\qquad \dim \A_I<\infty \;,
$$      
with $\mc{M}$ a closed Riemannian spin manifold. A basic idea is that every physical interaction comes from a suitable ``symmetry'' of the above almost commutative space: particle interactions from local gauge symmetries, and gravitational interactions from the symmetry under diffeomorphisms. It is natural to think that a first step in the unification of particle interactions with gravity is the unification of this two types of symmetries. The key for this unification is the split short exact sequence:
$$
1\longrightarrow\mathsf{Inn}(\A^\infty)\longrightarrow\mathsf{Aut}(\A^\infty)
\longrightarrow\mathsf{Out}(\A^\infty)\longrightarrow 1 \;.
$$
In the Einstein-Yang-Mills system, $\mathsf{Out}(\A_I)$ is trivial, so that
$\mathsf{Out}(\A^\infty)=\mathsf{Out}(C^\infty(\mc{M}))$ is isomorphic to $\mathrm{Diff}(\mc{M})$
and the automorphism group $\mathsf{Aut}(\A^\infty)$ (the group of symmetries of the full ``noncommutative space'') is a semidirect product of the group of diffeomorphisms of $\mc{M}$ with the group $\mathsf{Inn}(\A^\infty)=C^\infty(\mc{M}\to\mathsf{Inn}(\A_I))$
of smooth functions with values in $\mathsf{Inn}(\A_I)=U(\A_I)/Z_I$, where $Z_I$ is the center of $U(\A_I)$. The group $\mathsf{Inn}(\A^\infty)$ is what we call the local gauge group of the theory, while $\mathsf{Inn}(\A_I)$ is the global gauge group, or gauge group `tout court'.

On the other hand, the group $\mc{G}(\A^\infty,J)$ in \eqref{eq:one}
is isomorphic to the quotient $U(\A^\infty)/U(\widetilde{\A}_J)$, where $\widetilde{\A}_J:=\{a\in\A^\infty:aJ=Ja^*\}$ is a
$*$-subalgebra of the center of $\A^\infty$ \cite[Eq.~(2.3)]{walter}. One has $\mc{G}(\A^\infty,J)\supset\mathsf{Inn}(\A^\infty)$,
with equality if{}f $\widetilde{\A}_J$ is exactly the center of $\A$, and from $\mc{G}(\A^\infty,J)$ one recovers the local
gauge transformations of physics, while $\mc{G}(\A_I,J_I)$ gives the global ones.

For the Standard Model of particle physics the situation is slighly more complicated, and explained
in Proposition 1.199 of \cite{CM08}.

\smallskip

Given the importance of the group of gauge transformations in physics, it is very natural, in the framework of noncommutative geometry, to look for compact quantum group analogues of this notion. 
In fact, the idea of using quantum group symmetries to have a better understanding of the noncommutative geometric picture behind the Standard
Model was mentioned in several places by Connes; the problem of finding a nontrivial quantum group of symmetries of the finite space $F$ is posed on the last page of \cite{Con94}. It stimulated the program of Les Houches Summer School in Theoretical Physics in 1995, as documented in \cite{CGZ95} and motivated the study in \cite{Wan98a}.
In \cite{ludwik}, an approach along this line was made, where the quantum isometry group (in the sense of \cite{Gos10,BG09}) of the finite part of the Standard Model was computed. 
It was shown that its coaction, once extended to the whole spectral triple on $C^{\infty}(\mc{M})\otimes\A_F$,
leaves invariant both the bosonic and fermionic part of the spectral action,
thus providing us with genuine quantum symmetries of the Standard Model. In this article, we wish to continue the work in \cite{ludwik}  by investigating the notion of quantum gauge symmetries  which might be helpful in having a better understanding  of the noncommutative geometry approach to particle physics. Since in many of the applications to physics, the relevant algebra is the product of a commutative one with a finite noncommutative one described by a finite-dimensional $C^*$-algebra, we restrict our attention to \emph{finite-dimensional} $C^*$-algebras. On the other hand, we need to consider both complex and real $C^*$-algebras, since in one of the main applications of spectral triples to physics -- the Standard Model of elementary particles, the $C^*$-algebra involved is real.

\smallskip

It is evident that in order to have a correct quantum analogue of
\eqref{eq:one}, we first need to make sense of a compact quantum group version of the
unitary group of a finite dimensional (possibly real) $C^*$-algebra,
and then use it to define the quantum gauge group. It is natural to wonder
whether the free quantum groups $A_u ( n ) $ or their twisted counterparts
(denoted by $ A_u ( n, R ) $  in this article), first appearing in the
seminal works of Wang and Van Daele \cite{DW96,Wan95,Wan98a} can play the role of
quantum group of unitaries of $ M_n ( \C )$. The definition of these
compact quantum groups are recalled in Sec.~\ref{sec:cqgA}. The structure and
isomorphism classification of these quantum groups were studied in \cite{Wan98b}.
Since then, a considerable amount of literature has been developed around
these quantum groups (see e.g.~\cite{Ban05a,Ban05b,Bic03} for quantum symmetries of finite
metric spaces and graphs), which have also made contact with other branches
of mathematics, like combinatorics and free probability \cite{FreeC,FreeD}. 
We believe that the compact quantum group version of the unitary group is also important from the point of view of compact quantum group theory. Indeed, we will see that  we obtain  $A_u(n,R)$ as the quantum unitary group of $M_n(\C) $ whose adjoint action preserves the state $\tr( R^t\,.\,)$, where $R$ is any positive invertible $n\times n$ matrix. The dependence on $ R $ appears because unlike the classical case, a compact quantum group coaction on $ M_n ( \C ) $ does not need to preserve the usual trace. A byproduct of this construction  for real $ C^* $-algebras shows that a ``free'' analogue of the symplectic group $Sp(n)$ (quaternionic unitary group) can be realized as the quantum unitary group of the real $C^*$-algebra $M_n(\Q)$.

The plan of the paper is as follows. In Sec.~\ref{sec:cqg} we recall some necessary background about compact -- and in particular free -- quantum groups, spectral triples and real $C^*$-algebras.
In Sec.~\ref{sec:3}, inspired by the characterization of the group of unitaries of
a $C^*$-algebra $\A$ as the universal object in a certain category of
groups having a trace preserving action on $\A$, we define the quantum
analogue by passing to the category of \emph{quantum families}
(in the spirit of \cite{woro_pseudo,soltanquantumfamily})
and relaxing the condition of traciality of the state, that is
necessary in order to accommodate non-Kac type examples like $ A_u(n,R)$.
We prove that the universal object -- that we call \emph{quantum unitary group} -- exists and has a compact quantum group structure, by explicitly computing it for any finite-dimensional (complex and real) $C^*$-algebra.

In Sec.~\ref{sec:4}, we generalize the construction \eqref{eq:one} and define the quantum gauge group of a finite dimen\-sio\-nal spectral triple,
and compute it for three examples, namely the Einstein Yang Mills system, the spectral triple over the algebra $\A^{\mathrm{ev}}=\Q\oplus\Q\oplus M_4(\C)$
and for the spectral triple for the finite part of the Standard Model.
Finally, in Sec.~\ref{sec:5}, we discuss some aspects of  quantum symmetries of  finite-dimensional real $C^*$-algebras which were not dealt with, in \cite{ludwik}.
In particular, we prove the existence of quantum automorphism group for any finite-dimensional real $C^*$-algebra and prove that for matrix algebras
$M_n(\Bbbk)$, with $\Bbbk=\R,\C$ or $\Q$,
the quantum automorphism group coincides with the classical one.

Throughout the paper, by the symbol $\otimes_{\mathrm{alg}}$ we will always mean the algebraic tensor product over $\C$,
by $\otimes$ minimal tensor product of complex $C^*$-algebras or the completed tensor product of
Hilbert modules over complex $C^*$-algebras. The symbol $\otimes_{\R}$ will denote the tensor product
over the real numbers.
Unless otherwise stated, all algebras are assumed to be unital complex
associative involutive algebras.
We denote by $\mc{M}(\A)$ the  multiplier algebra of the complex $C^*$-algebra $\A$,
by $\mc{L}(\HH)$ the adjointable operators on the Hilbert module $\HH$
and by $\K(\HH)$ the compact operators on the Hilbert space $\HH$.
With the symbol $\{e_i\}_{1\leq i\leq n}$ we indicate the canonical orthonormal basis
of $\C^n$, with $\{e_{ij}\}_{1\leq i,j\leq n}$ the standard basis of $M_n ( \C )$
($e_{ij}$ is the matrix with $1$ in position $(i,j)$ and zero everywhere else), and
with $\mathbb{I}_n$ the $n\times n$ identity matrix.


\section{Compact quantum groups and spectral triples}\label{sec:cqg}

\subsection{Some generalities on compact quantum groups}\label{sec:cqgA}

We begin by recalling the definition of compact quantum groups and their coactions \cite{Wor87,Wor95}.

\begin{df}
A \emph{compact quantum group} (to be denoted by CQG from now on) is a pair $(Q,\Delta)$ given by a complex unital \mbox{$C^*$-algebra}
$Q$ and a unital $C^*$-homomorphism $\Delta:Q\to Q\otimes Q$ such that:
i)
$\Delta$ is coassociative, i.e.~
$
(\Delta\otimes\id)\circ\Delta=(\id\otimes\Delta)\circ\Delta
$
as equality of maps $Q\to Q\otimes Q\otimes Q$;
ii)
$\mathrm{Span}\bigl\{(a\otimes 1_Q)\Delta(b)\,\big|\,a,b\!\in\! Q\bigr\}$ and
$\,\mathrm{Span}\bigl\{(1_Q\otimes a)\Delta(b)\,\big|\,a,b\!\in\! Q\bigr\}$
are norm-dense in $Q\otimes Q$.
\end{df}

\noindent
For $Q=C(G)$, where $G$ is a compact topological group,
conditions i) and ii) correspond to the associativity and the cancellation
property of the product in $G$, respectively.

\begin{df}
A \emph{unitary corepresentation} of a compact quantum group $(Q,\Delta) $ on a Hilbert space $\HH$ is a unitary element $U\in\M(\K(\HH)\otimes Q)$
satisfying
$
\,(\id\otimes \Delta ) U = U_{(12)} U_{(13)} \,,
$
where we use the standard leg numbering notation (see e.g.~\cite{MD98}).
The corepresentation $U$ is \emph{faithful} if there is no proper $C^*$-subalgebra $Q'$ of $Q$ such that $U \in\M(\K(\HH)\otimes Q')$.
\end{df}

\noindent
If $Q=C(G)$, $U$ corresponds to a strongly continuous unitary representation of $G$.

For any compact quantum group $Q$ (see \cite{Wor87,Wor95}), there always exists a canonical dense \mbox{$*$-subalgebra} $Q_0\subset Q$ which is spanned by  the matrix coefficients of the finite dimensional unitary corepresentations of $Q$ and two maps
$\epsilon : Q_0 \to \C$ (counit) and  $\kappa : Q_0 \to Q_0$ (antipode) which make $Q_0$ a Hopf $*$-algebra.

\begin{df}
A \emph{Woronowicz $C^*$-ideal} of a CQG $(Q,\Delta)$ is a $C^*$-ideal $I$ of $Q$ such that $\Delta(I)\subset\ker(\pi_I\otimes\pi_I)$, where $\pi_I:Q\to Q/I$ is the quotient map.
The quotient $Q/I$ is a CQG with the induced coproduct.
\end{df}

If $Q=C(G)$ are continuous functions on a compact topological group $G$, closed subgroups of $G$ correspond to the quotients of $Q$ by its Woronowicz $C^*$-ideals. While quotients $Q/I$ give ``compact quantum subgroups'', $C^*$-subalgebras $Q'\subset Q$ such that $\Delta(Q')\subset Q'\otimes Q'$ describe ``quotient quantum groups''.

\begin{df}\label{def:alpha}
We say that a CQG $(Q,\Delta)$ coacts on a unital $C^*$-algebra $\A$
if there is a unital $C^*$-homomorphism (called a \emph{coaction})
$\alpha:\A\to\A\otimes Q$ such that:
i)
$(\alpha \otimes\id) \alpha=(\id\otimes \Delta) \alpha$;
ii)
$\mathrm{Span}\bigl\{\alpha(a)(1_\A \otimes b)
\,\big|\,a\in\A,\,b\in Q\bigr\}$ is norm-dense in $\A\otimes Q$.

The coaction is \emph{faithful} if any CQG $Q'\subset Q$ coacting on $\A$
coincides with $Q$.
\end{df}

\noindent
It is well known (cf.~\cite{Pod95,Wan98a}) that condition (ii) in Def.~\ref{def:alpha} is equivalent to the existence of a norm-dense unital $*$-subalgebra $\A_0$ of $\A$ such that the map $\alpha$, restricted to $\A_0$, gives a coaction of the Hopf algebra $Q_0$, that is to say: $\alpha(\A_0) \subset \A_0 \otimes_{\mathrm{alg}}Q_0$ and $(\id\otimes \epsilon)\alpha=id$ on $\A_0$.

For later use, let us now recall the concept of certain universal CQGs
defined in
\cite{DW96,Wan98b} and references therein.

\begin{df}\label{def:WangA}
For a fixed $n \times n$ positive invertible matrix $R$, $A_u(n,R)$ is the
universal $C^*$-algebra generated
by $\{u_{ij},\,i,j=1,\ldots,n\}$ such that
$$
uu^*=u^*u=\mathbb{I}_n\;,\qquad
u^t(R\hspace{1pt}\bar uR^{-1})=(R\hspace{1pt}\bar uR^{-1})u^t=\mathbb{I}_n\,,
$$
where $u:=((u_{ij}))$, $u^*: = (( u^*_{ji} )) $ and
$\,\bar u:=(u^*)^t$.
It is equipped with the 'matrix' coproduct $\Delta$ given on the generators
by
$$
\Delta(u_{ij})=\sum\nolimits_k u_{ik} \otimes u_{kj} \;.
$$
\end{df}
\noindent
Note that $u$ is a unitary corepresentation of $A_u(n,R)$ on $\C^n$.

The $A_u(n,R)$'s are universal in the sense that every compact \emph{matrix}
quantum group (i.e.~every CQG generated by the matrix entries of a
finite-dimensional
unitary corepresentation) is a quantum subgroup of $A_u(n,R)$ for some
$R>0,~n>0$ \cite{Wan98b}; in particular, the well-known quantum unitary
group $SU_q(n)$
is a quantum subgroup of some $A_u(n,R)$ (cf.~Sec.~\ref{rem:sun}).
It may also be noted that $A_u(n,R)$ is the universal object in the category
of CQGs which admit a unitary corepresentation on $\C^n$ such that the
adjoint coaction on the finite-dimensional $C^*$-algebra $M_n(\C)$ preserves
 the functional $M_n(\C)\ni m \mapsto \mathrm{Tr}(R^t m)$ (see
\cite{Wan99}).

More generally, for any invertible matrix $ F, $ an analogous construction can be done. 

\begin{df}[\cite{Ban96,Ban97}]\label{def:AoF}
Let $F\in GL_n(\C)$. A CQG denoted $A_u(n, F)$ is defined as the universal
$C^*$-algebra generated by $u=\{u_{ij}, \,i,j=1,\ldots,n\}$ with the condition
that both $u$ and $u'=F\hspace{1pt}\bar uF^{-1}$ are unitary;
equipped with the standard 'matrix' coproduct.
A quantum subgroup of $A_u(n, F)$, denoted by $A_o(n, F)$, is defined by
the additional relation $u=u'$.
\end{df}

One immediately realizes that $u'u'^*=F\hspace{1pt}\bar u
(F^*F)^{-1}u^tF^*=\mathbb{I}_n$ if and only if
$R\hspace{1pt}\bar u R^{-1}u^t=\mathbb{I}_n$ and
$u'^*u'=(F^*)^{-1}u^tF^*F\hspace{1pt}\bar u F^{-1}=\mathbb{I}_n$ if and only if
$u^tR\hspace{1pt}\bar u R^{-1}=\mathbb{I}_n$, where $R=F^*F$.
Thus $A_u(n, F)$ actually depend only on the modulus of $F$ and is isomorphic
to $A_u(n,R)$ for $R=F^*F$. Thus,  $A_o(n,F)$ is also a quantum subgroup of $A_u(n,R)$
for $R=F^*F$.

\smallskip
Since we will need both the quantum groups mentioned above, for clarity, we will use the symbol $ A_u (n,  F ) $ or $ A_o (n, F ) $ when $ F $ need not be a positive matrix and use $ R $ when it is positive.
Concerning the notation for free quantum orthogonal groups, we follow here that of \cite{Ban96},
which corresponds to $B_u(Q)$ in  \cite{Wan98b} for $Q=F^*$.
We refer to \cite{Wan98b} for a detailed discussion on the structure and
classification of such quantum groups.

We remark that the CQGs $A_u(n):=A_u(n,\mathbb{I}_n)$ and $A_o(n):=A_o(n,\mathbb{I}_n)$
are called the \emph{free quantum unitary group} and \emph{free quantum orthogonal group,
respectively}, as their quotient by the commutator ideal is respectively $C(U(n))$ and
$C(O(n))$.

\begin{rem}\label{rem:sp}
Let $n=2m$ be even and $F=\sigma_2\otimes \mathbb{I}_m$, where\vspace{-4pt}
\begin{equation}\label{eq:Pauli}
\sigma_2=
\bigg(\!\begin{array}{rr}
0 & \!-i \\ i & 0
\end{array}\!\bigg)\vspace{-4pt}
\end{equation}
is the second Pauli matrix and we identify $M_{2m}(\C)$ with $M_2(\C)\otimes M_m(\C)$.
In this case, the CQG $A_o(2m,F)$ will be denoted
$A_{sp}(m)$ and
it is a free version of the symplectic group $Sp(m)$
(the group of unitary elements of $M_m(\Q)$), that can be obtained as the
quotient of $A_{sp}(m)$
by the commutator subalgebra (cf.~Sec.~\ref{sec:3.5}).
We will see in Sec.~\ref{sec:3.5} that $A_{sp}(m)$ is the quantum unitary
group of $M_m(\Q)$.
The identification of $A_{sp}(m)$ with $A_o(2m,F)$ for a special $F$ was
pointed out to us by T.~Banica.
\end{rem}

A matrix $B$ (with entries in a unital $*$-algebra $\B$) such that
both $B$ and $B^t$ are unitary is called a \emph{biunitary}
\cite{BV09}. We will also need the following class of CQGs:

\begin{df}
For a fixed $n$, we call $A_u^*(n)$ the universal unital $C^*$-algebra generated
by an $n\times n$ biunitary $u=((u_{ij}))$ with relations
\begin{equation}\label{eq:fine}
ab^*c=cb^*a \;,\qquad\forall\;a,b,c\in\{ u_{ij},\,i,j=1,\ldots,n\} \;.
\end{equation}
$A_u^*(n)$ is a CQG with coproduct given by $\Delta(u_{ij})=\sum\nolimits_k u_{ik} \otimes u_{kj}$.
\end{df}

We will call $A_u^*(n)$ the $N$-dimensional \emph{half-liberated unitary group}.
This is similar to the half-liberated orthogonal group $A_o^*(n)$,
that can be obtained by imposing the further relation $a=a^*$ for all $a
\in\{ u_{ij},\,i,j,=1,\ldots,n\}$ (cf.~\cite{BV09}).

The analogue of projective unitary groups was introduced in \cite{Ban97}
(see also Sec.~3 of \cite{BV09}). Let us recall the definition.

\begin{df}\label{def:Ban}
Let $Q$ be a CQG which is generated by the matrix elements of a unitary corepresentation $U$.
The projective version $PQ$ of $Q$ is the Woronowicz $C^*$-subalgebra of $Q$ generated by the entries of $U\ot\bar{U}$(cf.~section 3 of \cite{BV09}).
In particular, $PA_u(n)$ is the $C^*$-subalgebra of $A_u(n)$ generated
by $\{u_{ij}(u_{kl})^*:i,j,k,l=1,\ldots,n\}$.
\end{df}

In \cite{Wan98a}, Wang defines the quantum automorphism group of $M_n(\C)$,
denoted by $A_{\mathrm{aut}}(M_n(\C))$ to be the universal object in the
category of CQGs with a coaction on $M_n(\C)$ preserving the trace
(and with morphisms given by CQGs homomorphisms intertwining the
coactions). The explicit definition is in Theorem 4.1 of \cite{Wan98a}.
We conclude this section by the following proposition stating
Theorem 1(iv) from \cite{Ban97} (cf.~also Prop.~3.1(3) of \cite{BV09})
and a very special case (namely, $q=1$) of Theorem 1.1 from \cite{Sol10}
together.

\begin{prop}
\label{thm:Ban}
We have $PA_u(n)\simeq PA_o(n)\simeq A_{\mathrm{aut}}(M_n(\C))$ and $ A_{\mathrm{aut}}(M_2(\C)) \simeq C ( SO ( 3 ) ). $ Thus, $ PA_u ( 2 ) \simeq PA_o ( 2 ) \simeq C ( SO ( 3 ) ). $
\end{prop}

\subsection{Relation between free unitary groups and $SU_q(n)$}\label{rem:sun}

In this section, we discuss the relation between the quantum unitary groups $A_u(n,R)$
and the quantum groups $SU_q(n)$ of \cite{FRTa,VS91,woro_tannaka}.
By the universality property, clearly $SU_q(n)$ is a quantum subgroup of 
$A_u(n',R)$ for a suitable $n'$ and $R$.
Our aim is pointing out explicitly how $SU_q(n)$ ``lies'' inside the
free quantum orthogonal group $A_u(n,R)$.

For $0<q\leq 1,$ we recall the definition of $ SU_q ( n ) $ following the notations of \cite[Sec.~9.2]{KS97}, except the fact that we will use $u_{ij}$ instead of $u^i_j$ to denote the matrix element of $u$ on the row $i$ and column $j$.
The CQG is generated by the matrix elements of an $n$-dimensional corepresentation $u=(u_{ij})$,
$i,j=1,\ldots,n$, with commutation relations
\begin{align*}
u_{ik}u_{jk} &=qu_{jk}u_{ik} &
u_{ki}u_{kj} &=qu_{kj}u_{ki} &&
\forall\;i<j\;, \\
[u_{il},u_{jk}]&=0 &
[u_{ik},u_{jl}]&=(q-q^{-1})u_{il}u_{jk} &&
\forall\;i<j,\;k<l\;,
\end{align*}
and with determinant relation
$$
\mc{D}_q=\sum\nolimits_{p\in S_n}(-q)^{||p||}
u_{1,p(1)}u_{2,p(2)}\ldots u_{n,p(n)}=1 \;,
$$
where the sum is over all permutations $p$ of the set
$\{1,2,\ldots,n\}$ and $||p||$ is the number of inversions
in $p$. The $*$-structure is given by
$$
(u_{ij})^*=(-q)^{j-i}\sum\nolimits_{p\in S_{n-1}}(-q)^{||p||}u_{k_1,p(l_1)}
u_{k_2,p(l_2)}\ldots u_{k_{n-1},p(l_{n-1})}
$$
with $\{k_1,\ldots,k_{n-1}\}=\{1,\ldots,n\}\smallsetminus\{i\}$,
$\{l_1,\ldots,l_{n-1}\}=\{1,\ldots,n\}\smallsetminus\{j\}$
(as ordered sets) and the sum is over all permutations $p$ of the set
$\{l_1,\ldots,l_{n-1}\}$.

From the defining relations, one derives the following `orthogonality' relations between rows
resp.~columns of $u$.
For all $a,b=1,\ldots,n$ we have:
\begin{align}
\sum\nolimits_iu_{ai}(u_{bi})^* &=\delta_{a,b} \;, &
\sum\nolimits_i(u_{ia})^*u_{ib} &=\delta_{a,b} \;, \label{eq:uustar} \\
\sum\nolimits_iq^{2(i-b)}u_{ia}(u_{ib})^* &=\delta_{a,b} \;, &
\sum\nolimits_iq^{2(a-i)}(u_{ai})^*u_{bi} &=\delta_{a,b} \;. \label{eq:ubarut}
\end{align}
This is simply Prop.~8 of \cite[Sec.~9.2.2]{KS97}, with quantum determinant $\mc{D}_q=1$
for $SU_q(n)$, and cofactor matrix $(-q)^{k-j}A^j_k=\tilde{u}^k_j=S(u^j_k),$
defined in page 313 of \cite{KS97} related to the real structure of $SU_q(n)$ by
the formula $u^*=S(u)=\tilde{u}^t$ (cf.~Sec.~9.2.4 of \cite{KS97}, case 2).

Now, equation \eqref{eq:uustar} in matrix form is simply
the unitarity condition $uu^*=u^*u=\mathbb{I}_n$. On the other hand, if we call
\begin{equation}\label{eq:RR}
R=\frac{1}{q^{n-1}[n]_q}\,\mathrm{diag}(1,q^2,q^4,\ldots,q^{2(n-1)}) \;,
\qquad [n]_q:=\frac{q^n-q^{-n}}{q-q^{-1}} \;,
\end{equation}
then
$$
(R\hspace{1pt}\bar uR^{-1})_{ij}=q^{2(i-j)}(u_{ij})^*
$$
and \eqref{eq:ubarut} is equivalent to the conditions
$u^t(R\hspace{1pt}\bar uR^{-1})=(R\hspace{1pt}\bar uR^{-1})u^t=\mathbb{I}_n$.
This was first noticed in \cite{Wan96},
cf.~eq.~(9) at page 675, and proves that $SU_q(n)$ is a quantum subgroup of the free unitary group
$A_u(n,R)$, for $R$ as in \eqref{eq:RR}. Clearly $R$ is not unique, for example one can multiply $R$ for
a constant, or replace $R$ with $R^{-1}$ ($SU_q(n)$ and $SU_{q^{-1}}(n)$ are
isomorphic, for any $q\in\R^+$).

For \mbox{$n=2$}, the corresponding state \mbox{$\varphi_R(a):=\tr(R a)$} is the well known Powers state of $M_2(\C)$.
This case was already dealt with in \cite{Wan98b}, cf.~page 578,
where it is proved that $SU_q(2)$ is isomorphic to $A_o(2,F)$ for
$$
F=\bigg(\!\begin{array}{cc}
0 & \!-q^{\frac{1}{2}} \\
q^{-\frac{1}{2}}\! & 0
\end{array}\bigg) \;.
$$
Clearly, $A_o(2,F)$ is a quantum subgroup of $A_u(2,R')$ for $R'=F^*F=\mathrm{diag}(q^{-1},q)$.
We refer to \cite{Wan99} for a study of the Power's state in relation to universal quantum groups,
and to \cite{Sol10} for its relation with $SO_q(3)$.

Notice that $[n]_qR=\pi(K_{2\rho})$, where $K_{2\rho}$ is the element of the dual Hopf $*$-algebra
$\mc{U}_q(\mathfrak{su}(n))$ implementing the modular automorphism (cf.~eq.~(3.2) of \cite{DD09})
and $\pi$ is the fundamental representation described in \cite[eq.~(4.1)]{DD09}.

\subsection{Generalities on real $C^*$-algebras}
We need to recall some basic facts about real $C^*$-algebras, which we are going to need throughout the article.
For more details on real $C^*$-algebras, we refer the reader to \cite{herbert} and \cite{goodearl}.
The following result characterizes all finite dimensional real $C^*$-algebras.

\begin{prop}\label{prop:classR}
Let $\A$ be a finite dimensional real $C^*$-algebra. Then $\A\cong M_{n_1}(D_1)\oplus M_{n_2}(D_2)\oplus M_{n_3}(D_3)\oplus\ldots\oplus M_{n_k}(D_k)$ (as real $C^*$-algebras) for some positive integers $n_1,n_2,\ldots n_k,$ where for each $i=1,2,\ldots k,$ $D_i$ is either $\R$, $\C$ or $\Q.$  
\end{prop}

\noindent
For a real $ C^* $-algebra $ \A,$ the $ \ast $-algebra $ \A_{\C} = \A \otimes_{\R} \C$ is a complex $C^*$-algebra, known as the complexification of $ \A. $ 
Moreover, $\A$ is the fixed point algebra of the antilinear automorphism
$\sigma$ on $\A_\C=\A\otimes_{\R}\C$, given by $\sigma(a\ot_{\R} z)=a\ot_{\R} \bar{z}.$
Note that $\sigma$ commutes with the involution on $\A_\C$, given by
$(a\otimes_{\R}z)^*=a^*\ot_{\R} \bar{z}$.
Throughout this article, the symbol $ \sigma $ will stand for this antilinear automorphism.

The following result recalls the complexifications and the formulas of $ \sigma $ for the finite dimensional $C^*$-algebras $ M_n ( \R ), ~ M_n ( \C ) $ and  $ M_n ( \Q ).$

\begin{prop}\label{sigmaforrealalgebrasformulae}
Let $\mc{A}:=M_n(\Bbbk),$ and $\mc{A}_\C:=\mc{A}\otimes_\R\C$. Then:\vspace{-5pt}
\begin{enumerate}\itemsep=0pt
\item
if $\Bbbk=\R,$ then $\mc{A}_\C=M_n(\C)$ and $\sigma(m)=\overline{m}$;
\item
if $\Bbbk=\C,$ then $\mc{A}_\C=M_n(\C)\oplus M_n(\C)$ and $\sigma(m_1\oplus m_2)=\overline{m_2}\oplus\overline{m_1}$;
\item
if $\Bbbk=\Q,$ then $\mc{A}_\C=M_2(\C)\otimes M_n(\C)\cong M_{2n}(\C)$ and $\sigma(m)=(\sigma_2\otimes 1_n)\overline{m}(\sigma_2\otimes 1_n)$, where
$\sigma_2$ is the matrix \eqref{eq:Pauli}.
\end{enumerate}
\end{prop}

\subsection{Real spectral triples}\label{sec:catCJ}

In noncommutative geometry, compact Riemannian spin manifolds are replaced by real spectral triples.
Recall that a unital spectral triple $(\A^\infty,\HH,D)$ is the datum of: a Hilbert
space $\HH$, a unital associative involutive
algebra $\A^\infty$ with a faithful unital $*$-representation $\pi:\A\to\B(\HH)$ (the representation
symbol is usually omitted) and a (not-necessarily bounded) self-adjoint operator $D$ on $\HH$ with
compact resolvent and having bounded commutators with all $a\in\A^\infty$, see e.g.~\cite{Con95,Con08}.
A spectral triple is \emph{even} if there is a $\Z_2$-grading $\gamma$ on $\HH$ commuting
with $\A^\infty$ and anticommuting with $D$; we will set $\gamma=1$ when the spectral triple is odd.
A spectral triple is \emph{real} if there is an antilinear isometry $J:\HH\to\HH$, called
the \emph{real structure}, such that
\begin{equation}\label{eq:yyy}
J^2=\epsilon 1\;,\qquad JD=\epsilon' DJ\;,\qquad
J\gamma=\epsilon''\gamma J\;,
\end{equation}
and
\begin{equation}\label{eq:real}
[a,JbJ^{-1}]=0\;,\qquad
[[D,a],JbJ^{-1}]=0\;,
\end{equation}
for all $a,b\in\A^\infty$~\footnote{%
In some examples (not in the present case) conditions \eqref{eq:real} have to be slightly relaxed, see e.g.~\cite{DLS05}.}.
$\epsilon$, $\epsilon'$ and $\epsilon''$ are signs and determine the KO-dimension of the space \cite{Con95}.

A canonical commutative example is given by $(C^\infty(M),L^2(\mc M,\mc S),\D)$ --
where $C^\infty(\mc M)$ are complex-valued smooth functions on a closed Riemannian
spin manifold, $L^2(\mc M,\mc S)$ is the Hilbert space of square integrable
spinors and $\D$ is the Dirac operator. This spectral triple is even if $\mc M$ is even-dimensional. In fact, from any commutative real spectral triple
it is possible to reconstruct a closed Riemannian spin manifold. We refer to \cite{Con08} for the exact statement.

While we always tacitly assume that $\HH$ is a complex Hilbert space,
we allow the possibility that $\A^\infty$ is a real $*$-algebra.
Note that to any real spectral triple 
$(\A^\infty,\HH,D,\gamma,J)$ over a real $*$-algebra $\A^\infty$, we can associate a real spectral triple $(\B^\infty,\HH,D,\gamma,J)$ over a complex $*$-algebra $\B^\infty$, as shown in Lemma 3.1 of \cite{ludwik}.
We let $\B^\infty$ be the quotient $\A^\infty_{\C}/\ker\pi_{\C}$, where
$\A^\infty_{\C}\simeq\A^\infty\otimes_{\R}\C$ is the complexification of $\A^\infty$,
with conjugation defined by $ ( a \otimes_{\R} z )^* = a^* \otimes_{\R} \bar{z}  $ for $ a \in \A^\infty $ and $ z \in \C$,
and $\pi_{\C}:\A^\infty_{\C}\to\B(\HH)$ is the $*$-representation
\begin{equation}\label{eq:piC}
\pi_{\C}(a\otimes_{\R}z)=z\pi(a) \;,\qquad a\in\A^\infty\,,\;z\in\C\,.
\end{equation}
It was observed in \cite{ludwik} that $\ker\pi_{\C} $ may be nontrivial since
the representation $\pi_{\C}$ is not always faithful. For example, if
$\A^\infty$ is itself a complex $*$-algebra (every complex $*$-algebra is also a
real $*$-algebra) and $\pi$ is complex linear, then for any $a\in\A^\infty$ the element
$a\otimes_{\R}1+ia\otimes_{\R}i\,$ of $\A^\infty_{\C}$
is in the kernel of $\pi_{\C}$. In fact, if $\A^\infty$ is a complex algebra,
$$
\B^\infty\simeq\A^\infty \;.
$$

We close this section with a remark. 
While usually $\A^\infty$ is only a pre-$C^*$-algebra for the operator
norm, in the finite-dimensional case it is a $C^*$-algebra, and to make this fact more evident it
will be denoted by $\A$, without the $\infty$ supscript.

\section{Quantum unitary group of a finite-dimensional $C^*$-algebra}\label{sec:3}

\subsection{The case of complex $C^*$-algebras}\label{sec:3.1}

Let $\A$ be a finite-dimensional complex $C^*$-algebra, that is,
\begin{equation}\label{eq:complexCstar}
\A=\bigoplus_{i=1}^m M_{n_i}(\C)
\end{equation}
for some positive integers $m$ and $n_i$.
For $a=a_1\oplus \ldots \oplus a_m\in\A$, we denote
by $\tr$ the trace map:
$$
\tr(a):=\sum_{i=1}^m\sum_{k=1}^{n_i}(a_i)_{kk} \;.
$$
Any faithful state of $\A$ is of the form $ \tr( R\,\cdot\, ) $ for some positive invertible operator $R:=\oplus_i R_i\in \A$ with normalization $\tr(R)=1$, called the \emph{density matrix} of the state. Since in the following,
the normalization of $R$ is irrelevant, in the particular case when $R=\frac{1}{\tr(\mathbb{I})}\mathbb{I}$ is 
a scalar multiple of the identity, one can equivalently use the map $\tr( \,\cdot\, ) .$
Let $ \pi_R: \A \rightarrow B ( L^2 ( A,  \tr( R\,\cdot\, ) ) $ be the GNS representation of a finite dimensional complex $ C^*  $-algebra $ \A $ with respect to the faithful state $\tr( R\,\cdot\, ) $ as above. We define the functional 
\begin{equation}\label{eq:varphiR}
\varphi_R(\pi_R ( a ) )=\tr( R a ) \;,
\end{equation}
The above functional is well defined since the GNS representation of a $ C^* $-algebra with respect to a faithful state is faithful. Throughout this article, the symbol $  \varphi_R$ will stand for this functional. 

We start by stating the following Lemma, which gives a characterization of the unitary group of a finite dimensional complex $C^*$-algebra.

\begin{lemma}
Let $\A$ be a finite dimensional complex $C^*$-algebra, viewed as a subalgebra of
$\B ( L^2 ( \A, \tr ) ) $ via the GNS representation $\pi$, and denote by
$\pi_U = \pi|_{U(\A)}$ its restriction to the group $U(\A)$ of unitary elements of $\A$.
Then $(U ( \A ), \pi_U )$ is the universal (final) object in the category
whose objects are pairs $(G,\tilde{\pi})$, with $G$ a compact group and
$\tilde{\pi}$ a unitary representation of $G$ on $L^2(\A,\tr)$ satisfying
$\tilde{\pi}(g)\in\A$ for all $g\in G$, and whose morphisms are continuous group homomorphisms intertwining the representations.
\end{lemma}

\begin{proof}
Clearly $(U ( \A ), \pi_U ) $ is an object in the category (as a linear space $L^2 ( \A, \tr )\simeq \A$ since the normalized trace is a faithful state, and then
$\pi$ is a faithful representation).
Moreover, if $(G,\tilde{\pi})$ is any object in the category,
since $\pi_U$ is faithful there exists a unique morphism $\phi:G\rightarrow U(\A)$
intertwining the representations, which is defined by $\phi(g)=(\pi_U)^{-1}\widetilde{\pi}(g)$
for all $g\in G$.
This shows the universality of $(U(\A),\pi_U)$.
\end{proof}

\noindent
We define a notion of quantum family of unitaries by taking a suitable noncommutative analogue of this characterization.
Notice that while $U(\A)$ is a \emph{final} object in the category described above, since the functor $C$ is contravariant, the $C^*$-algebra $C(U(\A))$ is a \emph{initial} object in the dual category.

\begin{df}\label{def:QUG}
Let $\A$ be a finite-dimensional complex $C^*$-algebra, $R\in\A$ a positive invertible operator,
$\varphi_R$ as in \eqref{eq:varphiR}, and let $\pi_R:\A\to\B\big(L^2(\A,\varphi_R)\big)$ be the
associated GNS representation.
We denote by $\mathbf{C_u}(\A,R)$ the category whose objects are pairs $(Q,U)$, 
with $Q$ a unital $C^*$-algebra and $U$ a unitary element in
$\pi_R(\A)\otimes Q$ such that:
\begin{enumerate}\itemsep=0pt
\item[(i)]  $\mathrm{Ad}_U=U(\,\cdot\,\otimes 1_Q)U^*$ preserves the state $\varphi_R$ on $\pi_R(\A)$,
\item[(ii)] $\mathrm{Ad}_{U^*}=U^*(\,\cdot\,\otimes 1_Q)U$ preserves the state $\varphi_{R^{-1}}$ on $\pi_R(\A)$,
\end{enumerate}
A morphism $\phi:(Q,U)\to (Q',U')$ is a $C^*$-homomorphisms such that $(\id\otimes\phi)(U)=U'$.

We call $\mathbf{C_u}(\A,R)$ the category of \emph{quantum families of $R$-unitaries} of $\A$.
\end{df}

\begin{rem}\label{rem:3.3}
Notice that condition (i) is equivalent to the condition that $U$ not only preserves
the inner product $\inner{a,b}_R=\varphi_R(a^*b)$ of the GNS representation (that follows from $U^*U=1$),
but also the sesquilinear form $(a,b)_R=\varphi_R(ab^*)$. If we consider the
subcategory whose objects $(Q,U)$ are compact matrix quantum groups, condition (ii) can be derived from (i) using the
properties of the antipode.
\end{rem}

\noindent
To start with, we will prove that the universal (initial) object in the category $\mathbf{C_u}(M_n(\C),R)$ exists and is in fact $A_u(n,R^t)$.
Using this result we will prove that for any finite-dimensional complex $C^*$-algebra $\A$,
 $\mathbf{C_u}(\A,R)$ has a universal object which is in fact a CQG.
We will call this CQG the \emph{quantum $R$-unitary group} of $\A$ and denote it by the symbol $Q_u( \A, R )$.

\begin{prop}\label{prop:QUGofMnC}
The universal object in the category $\mathbf{C_u}(M_n(\C),R)$ exists and it is isomorphic to
$(A_u(n,R^t),U_n)$, where $U_n$ is the faithful unitary corepresentation defined by
\begin{equation}\label{eq:UMnC}
U_n=\sum\nolimits_{i,j=1}^n\pi_R(e_{ij})\otimes u_{ij} \;,
\end{equation}
$u_{ij}$ are the canonical generators of $A_u(n,R^t)$ and $\pi_R:M_n(\C)\to 
\B\big(L^2(M_n(\C)),\varphi_R)\big)$ is the GNS representation.
\end{prop}

\begin{proof}
Since $\varphi_R$ is faithful, the linear space $L^2(M_n(\C),\varphi_R)$ is simply $M_n(\C)$ with inner product $\inner{a,b}\!{_R}=\varphi_R(a^*b)$. One can easily check that the map
$$
L_R:L^2(M_n(\C),\varphi_R)\to \C^n\otimes\C^n \;,\qquad
L_R(e_{ij})=e_i\otimes (R^t)^{\frac{1}{2}}e_j \;,
$$
is an isometry. Here, the inner product on $\C^n$ is the standard one, and $\{e_i\}_{i=1}^n$
is the canonical orthonormal basis of $ C^n. $ Moreover, we have
$$
L_R\pi_R(a)L_R^*=a\otimes \mathbb{I}_n
$$
so that a matrix $a\in M_n(\C)$ acts simply by row-by-column multiplication on the first factor $\C^n$.

Thus, for any object $(Q,V)$ in $\mathbf{C_u}(\A,R)$,
$V$ is of the form
$$
V=\sum\nolimits_{i,j=1}^n\pi_R(e_{ij})\otimes v_{ij}
=(L_R^*\otimes\id)\left(
\sum\nolimits_{i,j=1}^ne_{ij}\otimes 1\otimes v_{ij} 
\right)(L_R\otimes\id)
$$
with $v_{ij}\in Q$,
and unitarity of $V$ is equivalent to unitarity of the matrix $v\in M_n(Q)$.
Since $\varphi_R(e_{ij})=R_{ji}$ and
$$
\mathrm{Ad}_V(\pi_R(e_{ij}))=(L_R^*\otimes\id)\left(\sum\nolimits_{kl}
e_{kl}\otimes 1\otimes v_{ki}v_{lj}^*\right)(L_R\otimes\id),
$$
condition (i) in Def.~\ref{def:QUG} gives
$$
(\varphi_R\otimes\id)\mathrm{Ad}_V(\pi_R(e_{ij}))=\sum_{ijkl}R_{lk}v_{ki}v_{lj}^*=
\varphi_R(\pi_R(e_{ij}))\cdot 1_Q=R_{ji}\cdot 1_Q,
$$
that is, $v^tR^t\hspace{1pt}\bar{v}=R^t$. Similarly from (ii) of Def.~\ref{def:QUG}, we get
$\bar{v}\hspace{1pt}(R^t)^{-1}\hspace{1pt}v^t=(R^t)^{-1}$. These can be rewritten as
$$
v^tR^t\hspace{1pt}\bar{v}(R^t)^{-1}=
R^t\hspace{1pt}\bar{v}(R^t)^{-1}v^t=
\mathbb{I}_n \;.
$$
This proves that $v=(v_{ij})$ generate a quantum subgroup of $A_u(n,R^t)$.
It is clear from the above discussions that: \\[2pt]
1.~$(A_u(n,R^t),U_n)$ is an object of $\mathbf{C_u}(M_n(\C),R)$, with 
$U_n$ as in \eqref{eq:UMnC} (note that $U_n$ is a unitary corepresentation of $A_u(n,R^t)$,
and it is clearly faithful); \\[2pt]
2.~there is a unique $C^*$-homomorphism $\phi:A_u(n,R^t)\to Q$ such that $(\id\otimes\phi)(U_n)=V$.
This is uniquely defined by:
$$
\phi(u_{ij})=v_{ij} \;.
$$
This proves universality of $(A_u(n,R^t),U_n)$.
\end{proof}

This is very similar to a result of \cite{Wan98a}, except that here we consider
the category of quantum families (thus need the extra condition (ii) in Def.~\ref{def:QUG}, that is automatically satisfied when considering the category of CQGs)
and work with the GNS representation of the algebra.
We now extend this result to arbitrary finite-dimensional complex $C^*$-algebra.

\begin{thm}\label{thm:QUGA}
Let $\A=\bigoplus_{i=1}^m M_{n_i}(\C)$, $R=\oplus_kR_k\in\A$ a positive invertible operator,
and denote by $U_{n_k}$ the corepresentation of $A_u(n_k,R_k^t)$ on $L^2(M_{n_k}(\C),\varphi_{R_k})$
as in \eqref{eq:UMnC}.
Then the universal object $(Q_u(\A,R),U)$ in the category $\mathbf{C_u}(\A,R)$ exists and is given by
\begin{equation}\label{eq:QUG-uob}
Q_u(\A,R)=*_{k=1}^mA_u(n_k,R_k^t) \;,\qquad
U=\oplus_kU_{n_k} \;,
\end{equation}
where ``$*_{k=1}^m$'' is the free product
and $U$ is a faithful unitary corepresentation of $Q_u(\A,R)$ on $L^2(\A,\varphi_R)=\bigoplus_kL^2(M_{n_k}(\C),\varphi_{R_k})$.
\end{thm}

\begin{proof}
Firstly, we notice that $(*_{k=1}^mA_u(n_k,R_k^t),U)$ in \eqref{eq:QUG-uob} is an object of $\mathbf{C_u}(\A,R)$,
and the corepresentation $U$ is clearly faithful.

Let $\HH_k= L^2(M_{n_k}(\C),\varphi_{R_k}).$ Then,  $L^2(\A,\varphi_R)=\bigoplus_k\HH_k$.
For any object $(Q,V)$ of the category $\mathbf{C_u}(\A,R)$, we have
$$
V\in\pi(\A)\otimes Q=\oplus_k\pi_k(M_{n_k}(\C))\otimes Q
$$
with $\pi_k$ the GNS representation of $M_{n_k}(\C)$ on $\HH_k$.
Thus $V$ preserves $\HH_k$ and
 $V_k:=V|_{\HH_k}$ is a unitary in $ B ( \HH_k ) \otimes Q $ such that $V_k\in\pi_k(\A)\otimes Q$. Since $\mathrm{Ad}_{V_k}$ preserves
the state $\varphi_{R_k}$ and $\mathrm{Ad}_{V_k^*}$ preserves
the state $\varphi_{R_k^{-1}}$,
$(Q_k,V_k)$ is an object in $\mathbf{C_u}(M_{n_k}(\C),R_k)$, where $Q_k$ is the $C^*$-algebra
generated by the matrix entries of $V_k\in M_{n_k}(\C)\otimes Q$.

By Prop.~\ref{prop:QUGofMnC}, for every $k,$ there is a unique morphism $\phi_k:\big(A_u(n_k,R_k^t),U_{n_k}\big)\to (Q_k,V_k)$.
By universality of the free product, there is a unique morphism
$$
\phi\;:\;\big(*_{k=1}^mA_u(n_k,R_k^t),\oplus_kU_{n_k}\big)\to (Q,V)
$$
that restricted to the $k$-th factor gives $\phi_k$, and this is the unique $C^*$-homomorphism
from $*_{k=1}^mA_u(n_k,R_k^t)$ to $Q$ that intertwines the corepresentations $U$ and $V$.
This proves that the object $(*_{k=1}^mA_u(n_k,R_k^t),U)$ in \eqref{eq:QUG-uob} is universal in $\mathbf{C_u}(\A,R)$.
\end{proof}

We end this subsection by showing that the notion of quantum group of a nondegenerate bilnear form 
introduced in \cite{MDV90}  can be accomodated in our picture. For a non degenerate bilinear form given 
by an $ n \times n $ matrix $ B ,$ let $ Q^B $ be the universal algebra with generators 
$ ( q_{ij} )_{1 \leq i,j \leq n} $ satisfying the relations $ B^{- 1} q^t B q = I = q B^{- 1} q^t B, $ 
where $ q $ is the matrix $ (( q_{ij} ))_{ij} .$ Then the authors in \cite{MDV90} showed that $ Q^B $ 
has a Hopf algebra structure. We refer to \cite{MDV90} for the details. 
In Section 6 of \cite{Bic03a}, Bichon gave the necessary and sufficient conditions so that $ Q^B $ is a CQG.

For an $ n \times n $  nondegenerate matrix $ T $ and a complex algebra $ X $ we can define a $ X $ valued
 bilinear form on $ \C^n \otimes X $ by $ \left\langle \sum_i c_i \otimes x_{i} , \sum_j d_j \otimes y_{j} 
\right\rangle_{T} = \sum_{ij} T ( c_i, d_j  ) x_i y_j. $ 
Moreover, let $ V ( e_i ) = \sum e_j \otimes q_{ij} $ be 
a comodule coaction of a CQG $ Q $  on $ \C^n$, where $ e_i,\, i = 1,2,..., n$, is a basis of $ \C^n.$ 
Then the equation $ B^{- 1} q^t B q = I $ corresponds to  
$ \left\langle V ( x ), V ( y ) \right\rangle_{B} = \left\langle x, y \right\rangle_{B}\, 1, $ 
while $  q B^{- 1} q^t B = I $ corresponds to  $
 \left\langle V ( x ), V ( y ) \right\rangle_{B^{- 1}} = \left\langle x, y \right\rangle_{B^{- 1}}\,1. $ 

However, since in this article we are concerned with $ \ast $-algebras, it is more relevant to consider 
the following bilinear form on $\C^n \otimes Q $: 
$$
\left\langle \sum\nolimits_i c_i \otimes x_{i} , \sum\nolimits_j d_j \otimes y_{j} \right\rangle^{\prime}_{B} 
= \sum_{ij} B ( \bar{c_i}, d_j  ) x^*_i y_j,
$$
and similarly for $ B^{- 1}. $ 
Motivated by the above  observations, we slightly modify the definition of invariance of a bilinear form 
under a comodule coaction of a Hopf algebra as in \cite{MDV90} to give the following definition:

\begin{df}\label{duboisviolettedefinitionchanged}
For a nondegenerate positive definite bilinear form given by an $ n \times n $ matrix $ B, $ 
a unital $ C^* $ algebra $ Q $ and an element $ U \in M_n ( \C ) \otimes Q, $ the pair $ ( Q, U ) $ 
is said to preserve the bilinear form $ B $ if\vspace{-5pt}
\begin{list}{}{\itemsep=0pt \leftmargin=1.5em}
\item[(i)] $U$ is a unitary,
\item[(ii)] $\left\langle U ( v ), U ( w )  \right\rangle^{\prime}_{B} = \left\langle v, w  \right\rangle^{\prime}_{B}.1$,
\item[(iii)] $\left\langle U^* ( v ), U^* ( w ) \right\rangle^{\prime}_{B^{- 1}} 
=  \left\langle v, w \right\rangle^{\prime}_{B^{- 1}}.1$,
where $ v,w \in \C^n. $
\end{list}
\end{df} 

\noindent
With this definition at hand, the following proposition follows easily.

\begin{prop}
The category of a nondegenerate  positive definite bilinear form preserving quantum families, with objects 
$ ( Q, U ) $ as in Definition \ref{duboisviolettedefinitionchanged},
  has a universal object, which has a CQG structure isomorphic to $ A_u ( n, B^{- 1} ). $ 
Thus, the universal object which can be called the compact quantum group of a nondegenerate 
positive definite bilinear form $ B $ coincides with the quantum group of $ B^{- 1} $ unitaries of $ M_n ( \C ).$
\end{prop}  

\subsection{The case of real $C^*$-algebras}\label{sec:3.5}

In this section, we introduce the notion of quantum families of unitaries of a finite dimensional real $ C^* $-algebra by 
relating it to the quantum families of unitaries of its complexification.
As before, we start with an easy characterization of the group of unitaries for a finite dimensional real $ C^* $-algebra. We identify a real $C^*$-algebra $\A$ with the fixed point subalgebra of its complexification $\A_{\C}$ for a canonical
involutive antilinear automorphism $\sigma$. This in particular means that $U(\A)=\{u\in U(\A_{\C}):\sigma(u)=u\}$.
More generally, if $G\subset U(\A_\C)$ is a compact subgroup,
we have $G\subset U(\A)$ if and only if $\sigma(u)=u$ for all $u\in G$.
This can be rephrased in the dual language of corepresentations, using
the following observation.

\begin{prop} \label{rem:classUA}
Let $\A$ and $ \sigma $ be as above.
Let $T$ be an element in $\A_{\C}\otimes C(G)$, where $G$ is a compact group. Then
$( \id \otimes \varphi ) T  $ belongs to $\A $ 
for any state $ \varphi $ on $ C ( G )$
if and only if $ T = ( \sigma \otimes * ) T. $
\end{prop}
\begin{proof}
Since states separates points of a $C^*$-algebra,
$T = ( \sigma \otimes \ast ) T $ if and only if
 $ (\id\otimes \varphi ) T   =  (\id\otimes \varphi)( \sigma \otimes \ast ) T
 =\sigma(\id\otimes \varphi) T$ for all states $\varphi$,
where in the second equality we used $\varphi\circ *=*\circ\varphi$,
antilinearity of $\sigma$ and identify $\A_\C\otimes\C$ with $\A_\C$.
This is equivalent to $ ( \id \otimes \phi ) T \in \A$.
\end{proof}

\noindent
Inspired by this observation, we want to define a category of quantum families of unitaries of $\A$.
A first idea is to define such a category  as the subcategory of $\mathbf{C_u}(\A_{\C},R)$ 
whose objects $(Q,U)$ satisfy the additional condition
\begin{equation}\label{eq:toocomm}
(\sigma\otimes *)(U)=U \;.
\end{equation} 
It turns out that
this does not allow to accomodate non-Kac type examples. So we need to broaden the scope of our definition.
 The plan of this section is as follows: in Sec.~\ref{sebsec:one}, we explain what happens if we take 
\eqref{eq:toocomm}, illustrate it with some examples and then in Sec.~\ref{sebsec:two}, we define the 
category of quantum family of unitaries of a finite-dimensional real $C^*$-algebra and prove the existence 
of the universal object. Finally, In Sec.~\ref{sebsec:three}, we compute the universal object first for matrix
 algebras and then, using Prop.~\ref{prop:classR}, for any finite-dimensional real $C^*$-algebra.

\subsubsection{A preliminary study}\label{sebsec:one}
We want to explain that if we adopt \eqref{eq:toocomm}, we can get only Kac type examples as quantum
 unitary group.  As a first observation, we notice that \eqref{eq:toocomm} implies that $\mathrm{Ad}_U$
 preserves the usual trace. In fact, a more general statement is proved in the next lemma.

Here and in the following, we identify $\A_{\C}$ with the $C^*$-subalgebra 
$\A_{\C}\otimes 1_Q\subset\A_{\C}\otimes Q.$ 
Moreover, for an element $ T \in \A_{\C}, $ we write $ T $ instead of $ T \otimes 1 $ to simplify notations.

\begin{lemma}\label{lemma:usualTr}
Let $\A$ be a finite-dimensional real $C^*$-algebra, $F\in\A_{\C}$ an invertible element
and let $\pi:\A_{\C}\to\B\big(L^2(\A_{\C},\varphi_{\sigma(F^*F)})\big)$ be the GNS representation.
For  a (complex) unital $C^*$-algebra $Q,$
let $ U $ be a unitary element in  $\pi(\A_{\C})\otimes Q$ such that:
\begin{equation}\label{eq:based}
(\sigma\otimes *)(U)=F^{-1}UF \;,
\end{equation}
Let $ R = \sigma ( F^* F ). $ Then, $\varphi_R$ is preserved
by $\mathrm{Ad}_U$ and  $\varphi_{R^{-1}}$
is preserved by $\mathrm{Ad}_{U^*}$. In particular, taking $ F = 1, $ 
we deduce that the usual trace is preserved.
\end{lemma}

\begin{proof}
Let $ U = \sum_i U_{(1)i} \otimes U_{(2)i}. $ Equation \eqref{eq:based} implies
$ U = \sum_i \sigma ( F^{ - 1 } ) \sigma ( U_{(1)i} ) \sigma ( F ) \otimes U^*_{(2)i} $  
and $ U^* = \sum_j \sigma ( F^* ) \sigma ( U^*_{(1)j}   ) \otimes U_{(2)j}. $ 
Using these, we get:
\begin{align*}
(\varphi_R\otimes \id_Q &)\mathrm{Ad}_U(a) =
\sum\nolimits_{i,j} \tr ( R \sigma  ( F^{ - 1} ) \sigma ( U_{(1)i}  ) \sigma (  F ) a \sigma ( F^* ) ) \sigma ( U^*_{(1)j}  ) 
\sigma ( ( F^{- 1} )^*  ) ) U^*_{(2)i}  U_{(2)j}\\
&= \sum\nolimits_{i,j} \tr (  \sigma ( ( F^{- 1} )^*  ) R \sigma  ( F^{ - 1} ) \sigma ( U_{(1)i}  ) \sigma (  F ) a \sigma ( F^* ) ) 
\sigma ( U^*_{(1)j}  )  ) U^*_{(2)i}  U_{(2)j}\\
&=\sum\nolimits_{i,j} \tr ( \sigma ( U^*_{(1)j} ) \sigma ( U_{(1)i} ) \sigma ( F ) a \sigma ( F^* )    ) U^*_{(2)i}  U_{(2)j}\\
&=( \tr \otimes \id_Q ) ( \sigma \otimes \ast ) \left( \left( \sum\nolimits_j U^*_{(1)j} \otimes   U^*_{(2)j}\right) \left( \sum\nolimits_i U_{(1)i}
 \otimes   U_{(2)i}\right)   \right) \big(\sigma ( F ) a \sigma ( F^* ) \otimes 1 \big) \\
&=\tr (  \sigma ( F ) a \sigma ( F^* ) ). 1_Q \; ,
\end{align*}
where we have used the unitarity of $U$. Therefore, $\varphi_R$ is preserved
by $\mathrm{Ad}_U$. Similarly one shows that $\varphi_{R^{-1}}$
is preserved by $\mathrm{Ad}_{U^*}$.
\end{proof}

If a CQG has a unitary corepresentation on $\C^n$ such that its adjoint coaction
on $M_n(\C)$ preserves the trace, then it is a quotient of $A_u(n)$, that is known to be a Kac algebra 
(i.e.~the square of the antipode is the identity). Therefore, in order to obtain non-Kac algebras, 
we need to relax this condition.
Notice that the above phenomenon is a purely quantum phenomenon, since
a unitary group action always preserves the trace. In fact, one can show that:

\begin{prop}\label{prop:3.9}
Let $ \A   $ be a real $ C^* $-algebra such that $ \A_{\C} = \oplus^m_{i = 1} M_{n_i} ( \C ) $, and let
$ R = \oplus^m_{i = 1} R_i \in \A_{\C} $ be a positive invertible operator such that each $ R_k $ has $n_k$ distinct 
eigenvalues. Consider the subcategory of $\mathbf{C_u}(\A_{\C},R)$ whose objects $(Q,U)$ satisfy the 
additional condition \eqref{eq:toocomm}. Then, any CQG in this category with a faithful corepresentation 
is a quotient of a free product $*_{k=1}^nC(U(1))$ with $n=\sum_{i=1}^mn_i$.
\end{prop}
\begin{proof}
Let $(Q,U)$ be an object of $\mathbf{C_u}(\A_{\C},R)$ with $ Q $ a CQG and a faithful corepresentation 
$U$ satisfying condition \eqref{eq:toocomm}. Let $ U_k = U|_{\HH_k}$, 
where $ \HH_k = L^2 ( M_{n_k} ( \C ), \varphi_{R_k} ). $ By using the proof of Theorem \ref{thm:QUGA}, 
we deduce that $ Q $ is a quotient of $ \ast^m_{k = 1} A_u ( n_k, R^t_k ) $ and $ U = \oplus_k U_k. $ 
Moreover, for each $ k, $ we have 
$$
U^t_k R_k^t\bar{ U_k}=R_k^t \;.
$$
By Lemma 2.1 of \cite{DW96}, we can assume that each $R_k$ is a diagonal matrix.
However, by  Lemma \ref{lemma:usualTr} and the proof of Proposition \ref{prop:QUGofMnC}, 
we have $\bar{ U_k} U^t_k= U^t_k \bar{ U_k} = \mathbb{I}_{n_k}$ for all $ k. $
Using $\bar U U^t=\mathbb{I}_n$ on the previous equation, we get
$
R^t_k\bar{ U_k}=\bar{ U_k}R^t_k
$.
Since $R^t_k$ is diagonal with distinct entries, $\bar{ U_k}$ commutes with $R^t_k$ if and only if
it is diagonal too. Thus
$U_k=\mathrm{diag}(u_1,u_2,\ldots,u_{n_k})$, where each $u_i$ generates a copy of $C(U(1))$.
Hence $Q$ is a quotient of $*_{k=1}^nC(U(1))$ with $n=\sum_{k=1}^mn_k$.
\end{proof}

The previous proposition applies, for example, to $R$ as in \eqref{eq:RR},
when $q\neq 1$. Moreover, applying this result to the real $ C^* $-algebra  
$\A=\Q$ with $R\in\A_{\C}=M_2(\C)$ the density matrix of the Powers state, we have:

\begin{cor}
Let $0<q\leq 1$ and $R=[2]_q^{-1}\,\mathrm{diag}(q^{-1},q)$.
Consider the subcategory of $\mathbf{C_u}(M_2(\C),R)$ whose objects $(Q,U)$
satisfy the additional condition \eqref{eq:toocomm},
with $\sigma(m)=\sigma_2\overline{m}\sigma_2$
as in Prop.~\ref{sigmaforrealalgebrasformulae}, with $n=1$. Then any CQG in this category 
is a quotient of $ C ( U ( 1 ) ). $
\end{cor}

The classical
group of unitary elements of the real $C^*$-algebra $\Q$ is $ SU ( 2 ). $  Thus, for the above choice
 of $R,$ the CQG
that we get is neither a deformation of the classical unitary group nor does it contain a deformation of it. 
Thus the 
condition \eqref{eq:toocomm} is evidently too restrictive. However, we get a much better result 
if we change it slightly.

\begin{prop}\label{prop:qugH}
Let $R=[2]_q^{-1}\,\mathrm{diag}(q^{-1},q)$.
Consider the subcategory of $\mathbf{C_u}(M_2(\C),R)$ whose objects $(Q,U)$
satisfy the additional condition:
\begin{equation}\label{eq:correctcond}
(\sigma\otimes *)(U)=R^{\frac{1}{2}}UR^{-\frac{1}{2}}
\end{equation}
with $\sigma$ as in Prop.~\ref{sigmaforrealalgebrasformulae} (with $n=1$).
The universal object in this category is $SU_q(2)$, for any $0<q\leq 1$. Taking $ q = 1 ,$
 we recover the fact that the classical unitary group is $ SU ( 2 ). $ 
\end{prop}
\begin{proof}
If $(Q,U)$ is an object in the above-mentioned category, $U=(u_{ij})\in M_2(Q)$, 
from \eqref{eq:toocomm} we get:
$$
(\sigma\otimes *)(U)=
\bigg(\!\begin{array}{rr}
u^*_{22} & -u^*_{21} \\
-u^*_{12} & u^*_{11}
\end{array}\!\bigg)
=R^{\frac{1}{2}}UR^{-\frac{1}{2}}=
\bigg(\!\begin{array}{cc}
u_{11} & q^{-1}u_{12} \\
qu_{21} & u_{22}
\end{array}\!\bigg) \;.
$$
Hence
$$
U=\bigg(\!\begin{array}{cc}
a & -qc^* \\
c & a^*
\end{array}\!\bigg) \;.
$$
By Theorem \ref{thm:QUGA}, $U$ satisfies the relations of $A_u(2,R)$. It is an easy
computation to check that these relations are satisfied if and only if $a,c$ satisfy
the defining relations of $SU_q(2)$, in the notations of \cite[Sec.~4.1.4]{KS97}.
The remaining conditions
$U^t(R\hspace{1pt}\bar{U}R^{-1})=(R\hspace{1pt}\bar{U}R^{-1})U^t=\mathbb{I}_2$
are automatically satisfied. It follows from the above discussion that $SU_q(2)$
-- with generators denoted $a',c'$ -- is an object in the category, and there is
a unique $C^*$-homomorphism from $SU_q(2)$ to $Q$ intertwining the corepresentation,
given by $a'\mapsto a$ and $c'\mapsto c$. This proves universality.
\end{proof}

\subsubsection{Definition and existence of quantum unitary group of real $C^*$-algebras}\label{sebsec:two}

Based on the observations made in the last section, we  modify \eqref{eq:toocomm} to define the quantum family of unitaries of a finite dimensional real $ C^* $-algebra. In particular
the condition \eqref{eq:rela} is just \eqref{eq:based}.

\begin{df}\label{def:3.32}
Let $\A$ be a finite-dimensional real $C^*$-algebra, $F\in\A_{\C}$ an invertible element
and let $\pi:\A_{\C}\to\B\big(L^2(\A_{\C},  \varphi_{\sigma(F^*F)})  \big)$ be the GNS representation.
We denote by $\mathbf{C_{u,\R}}(\A,F)$ the category 
whose objects $(Q,U)$ are given by a (complex) unital $C^*$-algebra $Q$ and
a unitary element $U\in\pi(\A_{\C})\otimes Q$ such that:
\begin{equation}\label{eq:rela}
(\sigma\otimes *)(U)=F^{-1}UF \;,
\end{equation}
A morphism $\phi:(Q,U)\to (Q',U')$ is a $C^*$-homomorphisms such that $(\id\otimes\phi)(U)=U'$.

We call $\mathbf{C_{u,\R}}(\A,F)$ the category of \emph{quantum families of
$F$-unitaries} of $\A$.
\end{df}

The condition \eqref{eq:rela} is inspired by \eqref{eq:correctcond} (where $F=R^{-\frac{1}{2}}$).

As an immediate corollary to Lemma \ref{lemma:usualTr}, we obtain:
\begin{lemma} \label{newconditiongivesnewR}
For any $\A$ and $F$ there is a positive invertible element $R\in\A_{\C}$, given by
$R= \sigma ( F^*F )$, such that $\mathbf{C_{u,\R}}(\A,F)$ is a subcategory of $\mathbf{C_u}(\A_{\C},R)$.
\end{lemma}

The next lemma will be needed later to prove that if the
universal object in $\mathbf{C_{u,\R}}(\A,F)$ exists, then it is a CQG. 

\begin{lemma} \label{realalgebraworidealoctober}
If $Q$ is any CQG with a corepresentation $U\in\A_{\C}\otimes Q$,
the ideal $I_F$ generated by the relation \eqref{eq:rela} is a Woronowicz $C^*$-ideal.
\end{lemma}
\begin{proof}
In this proof, we write explicitly $F\otimes 1$ for the element in $\A_{\C}\otimes Q$ (instead of $F$,
with the usual identification of $\A_{\C}$ with its image in $\A_{\C}\otimes Q$) to make the proof
more transparent. Let
$$
T:=(\sigma\otimes *)(U)-(F^{-1}\otimes 1)U(F\otimes 1) \;\in\A_{\C}\otimes Q \;.
$$
The ideal $I$ is generated by
$$
t_\varphi= (\varphi\otimes\id)(T) \;,
\qquad
\varphi\in (\A_{\C})^*\;.
$$
Let $\pi_I:Q\to Q/I$ be the quotient map.
We need to prove that
$$
(\pi_I\otimes\pi_I)\Delta(t_\varphi)=
(\varphi\otimes\pi_I\otimes\pi_I)(\id\otimes\Delta)(T)
$$
is zero for all $\varphi$. Hence,  it is enough to prove that
$(\id\otimes\pi_I\otimes\pi_I)(\id\otimes\Delta)(T)=0$.

We have
$$
(\id\otimes\Delta)(\sigma\otimes *)(U)=
(\sigma\otimes *\otimes *)U_{(12)}U_{(13)}
=(\sigma\otimes *\otimes *)(U_{(12)})
\cdot
(\sigma\otimes *\otimes *)(U_{(13)}).
$$
Moreover, we notice that
\begin{multline*}
(\id\otimes\Delta)(F^{-1}\otimes 1)U(F\otimes 1) =
(F^{-1}\otimes 1\otimes 1)(\id\otimes\Delta)(U)(F\otimes 1\otimes 1)
\\
=
(F^{-1}\otimes 1\otimes 1)U_{(12)}(F\otimes 1\otimes 1)
(F^{-1}\otimes 1\otimes 1)U_{(13)}(F\otimes 1\otimes 1)
\\
=\big\{(F^{-1}\otimes 1)U(F\otimes 1)\big\}_{(12)}
\big\{(F^{-1}\otimes 1)U(F\otimes 1)\big\}_{(13)} \;.
\end{multline*}
Since $id\otimes\pi_I\otimes\pi_I$ is a $C^*$-algebra
morphism, it is enough to prove that the elements
$$
(\sigma\otimes *\otimes *)(U_{(12)})-\big\{(F^{-1}\otimes 1)U(F\otimes 1)\big\}_{(12)}
$$
and
$$
(\sigma\otimes *\otimes *)(U_{(13)})-\big\{(F^{-1}\otimes 1)U(F\otimes 1)\big\}_{(13)}
$$
are in the kernel of $id\otimes\pi_I\otimes\pi_I$ (if $a-b$ and $c-d$ are in the kernel of a morphism, then
$ac-bd=(a-b)c+b(c-d)$ is in the kernel too).
But this follows easily from \eqref{eq:rela}. Indeed:
\begin{multline*}
(\id\otimes\pi_I\otimes\pi_I)(\sigma\otimes *\otimes *)(U_{(12)}) =
\big\{(\id\otimes\pi_I)(\sigma\otimes *)(U)\big\}_{(12)} \\[2pt]
=\big\{(\id\otimes\pi_I)(F^{-1}\otimes 1)U(F\otimes 1)\big\}_{(12)} \\
=(\id\otimes\pi_I\otimes\pi_I)(F^{-1}\otimes 1\otimes 1)U_{(12)}(F\otimes 1\otimes 1).
\end{multline*}
The other equality  for $(\sigma\otimes *\otimes *)(U_{(13)})$ follows similarly.
This concludes the proof.
\end{proof}

\begin{thm}
The universal object of $\mathbf{C_{u,\R}}(\A,F)$, denoted by $Q_{u,\R}(\A,F)$, exists
and is the CQG given by $Q_u(\A_{\C},R)/I_F$, with $\sigma(R)=F^*F$ and $I_F$ the Woronowicz $C^*$-ideal generated by the relation \eqref{eq:rela}.
\end{thm}
\begin{proof}
Clearly $Q_u(\A_{\C},R)/I_F$ is an object in the category $\mathbf{C_{u,\R}}(\A,F)$. On the other hand, by Lemma \ref{newconditiongivesnewR},
every object in $\mathbf{C_{u,\R}}(\A,F)$ is an element $(Q,U)\in \mathbf{C_u}(\A_{\C},R)$
satisfying \eqref{eq:rela}. 
Since $Q_u(\A_{\C},R)$ is universal in  $ \mathbf{C_u}(\A_{\C},R),$  there is a unique morphism 
$\phi:Q_u(\A_{\C},R)\to Q$ in the category $ \mathbf{C_u}(\A_{\C},R)$. Since $Q$
satisfies \eqref{eq:rela}, 
there exists a map $ \psi: Q_u(\A_{\C},R)\to Q_u(\A_{\C},R)/I_F\to Q,$ such that  $ \phi = \psi \circ \pi_{I_F} ,$  $ \pi_{I_F} $ being the quotient map from $Q_u(\A_{\C},R) $ to $Q_u(\A_{\C},R)/I_F.$ Suppose that there exists another morphism $ \psi^{\prime} $ from $ Q_u(\A_{\C},R) $ to $ Q_u(\A_{\C},R)/I_F\to Q $ 
in the category $\mathbf{C_{u,\R}}(\A,F).$ Then $ \psi^{\prime} \circ \pi_{I_F} $ is another morphism from $ Q_u(\A_{\C},R)\to Q$ in the category $ \mathbf{C_u}(\A_{\C},R),$ contradicting the uniqueness of $ \phi. $ This proves that $Q_u(\A_{\C},R)/I_F$ is the universal object in $\mathbf{C_{u,\R}}(\A,F),$ which is a CQG due to Lemma \ref{realalgebraworidealoctober}.
\end{proof}

\subsubsection{Examples}\label{sebsec:three}

In Prop.~\ref{prop:qugH} we proved that when $R$ is the density matrix of the Powers state,
$Q_{u,\R}(\Q,R^{\frac{1}{2}})$ is the quantum group $SU_q(2)$. Let us extend the computation to $M_n(\R)$, $M_n(\C)$ and $M_n(\Q)$.

\begin{prop}\label{prop:FKH}
Let \mbox{$F=K\oplus H\bar K^{-1}$}, with \mbox{$H,K\in GL(n,\C)$}.
Then $Q_{u,\R}(M_n(\C),F)$ is the quotient of $A_u(n,K^*K)$ by the relation $uH=Hu$.
In particular for $H=\mathbb{I}_n$, $Q_{u,\R}(M_n(\C),F)\simeq A_u(n,R)$, with $R=K^*K$.
\end{prop}
\begin{proof}
Let $(Q,U)$ be any object of $\mathbf{C_{u,\R}}(M_n(\C),F)$,
with $U=u_1\oplus u_2$ and $u_i\in M_n(\C)\otimes Q$ for $i=1,2$.
Condition \eqref{eq:rela} is equivalent to
$$
\bar u_2\oplus\bar u_1=K^{-1}u_1K\oplus \bar KH^{-1}u_2H\bar K^{-1} \;,
$$
that is $u_1=K\bar u_2K^{-1}$ and
$\bar u_1=\bar KH^{-1}u_2H\bar K^{-1}$. 
Conjugating the second equation we see that
$u_1=K\bar u_2K^{-1}=K\overline{H^{-1}u_2H}K^{-1}$,
that implies $H^{-1}u_2H=u_2$.
Thus $Q$ is generated by the
matrix entries $v_{ij}$ of $u_2$ with the condition that
$$
U=K\hspace{1pt}\bar{u}_2K^{-1}\oplus u_2
$$
and that $u_2\in M_n(Q)$ commutes with $H$.
The operator $U$ is unitary
if and only if both $u_2$ and $K\hspace{1pt}\bar{u}_2K^{-1}$
are unitary, i.e.~by Def.~\ref{def:AoF} the elements
$v_{ij}$ satisfy the defining relations of $A_u(n,R)$,
with $R=K^*K$.

It is clear from the above discussion that $A_u(n,R)/I_H$,
where $I_H$ is the ideal generated by the relation $uH=Hu$
and $u=(u_{ij})$ the canonical
generators, is an object in the above category. Moreover,
there is a unique $C^*$-homorphism $A_u(n,R)/I_H\to Q$
intertwining the corepresentations, given by $u_{ij}\mapsto v_{ij}$.
This proves that the object $A_u(n,R)/I_H$ is universal.
\end{proof}

\noindent
Notice that in previous proposition we consider the most general invertible \mbox{$F\in M_n(\C)\oplus M_n(\C)$},
that without loss of generality can be written as $F=K\oplus H\bar K^{-1}$.

\begin{prop}
$Q_{u,\R}(M_n(\R),F)$ is isomorphic to $A_o(n,F)$.
$Q_{u,\R}(M_k(\Q),F)$ is isomorphic to $A_o\bigl(2k,F(\sigma_2\otimes \mathbb{I}_k)\bigr)$.
In particular, $Q_{u,\R}(M_n(\R),I)$ is the free quantum orthogonal group $A_o(n)$
and $Q_{u,\R}(M_k(\Q),I)$ is the free quantum symplectic group $A_{sp}(k)$.
\end{prop}

\begin{proof}
The involutions $\sigma$ for $M_n(\R)$ and $M_k(\Q)$ are given in Prop.~\ref{sigmaforrealalgebrasformulae}. $M_n(\R)$ and $M_k(\Q)$ are respectively the fixed point
real subalgebras of $M_n(\C)$ and $ M_{2k} ( \C ) $  for the automorphism $\sigma$ defined by $\sigma(a)=K\hspace{1pt}\bar{a}K^*,$ where $K=\mathbb{I}_n$ for $ M_n ( \C ) $ and $K=\sigma_2\otimes \mathbb{I}_k$ for $M_k(\Q).$
Here, as usual, we identify $M_{2k}(\C)$ with $M_2(\C)\otimes M_k(\C)$.
In both the cases, the condition 
\eqref{eq:rela} becomes $\bar U=(FK)^{-1}U(FK),$ using which the proposition follows easily.
\end{proof}

\noindent
Like the complex case, if we have a direct sum of algebras we get a free product
of CQGs, i.e.%
$$
Q_{u,\R}(\A_1\oplus\A_2,R_1\oplus R_2,F_1\oplus F_2)=Q_{u,\R}(\A_1,R_1,F_1)*Q_{u,\R}(\A_2,R_2,F_2) \;.
$$
The proof is analogous to the one of Theorem \ref{thm:QUGA}, and we omit it.

We conclude this section by identifying the quantum group of unitaries of the two algebras $\A_F= \C \oplus \Q \oplus M_3 ( \C ) $ and
$\A^{\mathrm{ev}}=\Q \oplus \Q \oplus M_4 ( \C )$ which appear in the noncommutative geometry formulation of the Standard Model.

\begin{cor}
The quantum unitary groups of the real $C^*$-algebras $ \C \oplus \Q \oplus M_3 ( \C ) $ and $ \Q \oplus \Q \oplus M_4 ( \C ) $ are $ C(U ( 1 )) \ast C ( SU ( 2 ) ) \ast A_u ( 3 ) $ and $ C ( SU ( 2 ) ) \ast C ( SU ( 2 ) ) \ast A_u ( 4 ) $, respectively.
\end{cor} 

\section{Quantum gauge group of a finite-dimensional spectral triple}\label{sec:4}
In this section we will define a quantum analogue of the gauge group \eqref{eq:one}.
As explained in the introduction, for physical reasons, we are interested
in the finite part of an almost commutative spectral triple: in this case
\eqref{eq:one} is the ``global'' gauge group of the theory.
We will focus, then, on finite-dimensional (real) spectral triples $(\A,\HH,D,J)$.
This means that $\HH$ is a finite-dimensional Hilbert space and $\A$ a finite-dimensional (possibly real) $C^*$-algebra. In the construction, the operator $D$ is irrelevant and we will assume that, even when $\A$ is real, $\HH$ is a complex Hilbert space (cf.~Sec.~\ref{sec:catCJ}).

We define the quantum gauge group using only quantum $R$-unitaries with $R\propto\mathbb{I}$. This is the most interesting case, since one gets
the classical gauge group as a quantum subgroup. The construction can be
adapted to the general case with minor modifications.

\begin{df}[\cite{CM08,walter}]
The gauge group of a finite-dimensional spectral triple $(\A,\HH,D,J)$ is the group
$$
\mc{G}(\A,J):=\{v:=uJuJ^{-1}:u\in U(\A)\} \;,
$$
with $U(\A)$ the unitary group of $\A$.
\end{df}

Now we propose %
a definition of the quantum gauge group for a finite-dimensional spectral triple over a real $C^*$-algebra. The complex case is easier, and
follows with some obvious changes. %
Throughout this section, we will use \eqref{eq:yyy} and \eqref{eq:real}, sometimes without mentioning it.

As explained in Sec.~\ref{sec:catCJ}, to any finite-dimensional real spectral triple
$(\A,\HH,D,\gamma,J)$ over a real $C^*$-algebra $\A$, we 
can associate a finite-dimensional real spectral triple $(\B,\HH,D,\gamma,J)$ over the
complex $C^*$-algebra $\B:=\A_{\C}/\ker\pi_{\C}$, where $\pi_{\C}$ is the $*$-representation \eqref{eq:piC}.

We need some preliminary observations.

\begin{lemma}\label{newlemmaforgaugegroupputbyme}
~\vspace{-5pt}
\begin{enumerate}\itemsep=0pt
\item Let us identify $a\in\A$ with $a\otimes_{\R}1$ in $\A_\C$. Then:
\begin{equation}\label{images are equal}
\B=
\pi_\C(\A_\C)
=\mathrm{Span}\bigl\{za\,:\,z\in\C,a\in\pi(\A)\big\} \;.
\end{equation}
\item For any $a\in\pi_\C(\A_\C)$ we have:
\begin{equation}\label{commute kore}
J\pi_\C(a)J^{-1}\in\pi_\C(\A_\C)' \;.
\end{equation}
\item
Let $U\in\A\ot Q$ be a unitary corepresentation of a CQG $Q$ on the Hilbert space $\HH$, say $U=\sum_{k=1}^ra_k\otimes q_k$ for some $r\geq 1$, $a_k\in\A$ and $q_k\in Q$, for all $k=1,\ldots,r$. Then
$$
\bar{U}=\sum\nolimits_{k=1}^r\,\bar{a_k}\ot q^*_k \;,
$$
where ``bar'' indicates the conjugated of a matrix in any fixed basis of $\HH$.
\end{enumerate}
\end{lemma}

\begin{proof}
The equation \eqref{images are equal} follows from the definition of $\pi_\C.$
For \eqref{commute kore}, we will use \eqref{images are equal}.
Let $za$ and $wb$ two elements of $\pi_{\C}(\A_{\C})$, with $z,w\in\C$
and $a,b\in\pi(\A)$.
Now $J za J^{-1}=\bar{z}JaJ^{-1},$ since $J(\cdot)J^{-1}$ is antilinear. Hence, $JzaJ^{-1}wb=w\bar{z}JaJ^{-1}b=w\bar{z}bJaJ^{-1}=wbJzaJ^{-1},$ since $JzaJ^{-1}$ is a linear operator on $\B(\HH).$ So, $wb$
commutes with $J za J^{-1}$, and this is extended to
arbitrary elements $\sum_iz_ia_i$ and $\sum_iw_ib_i$
of $\pi_{\C}(\A_{\C})$ by bilinearity of the commutator,
proving \eqref{commute kore}.

To prove 3, we fix an orthonormal basis $\{e_i\}_{i=1}^n$ of $\HH$, with $n=\dim_{\C}(\HH)$. We denote by $e_{ij}\in\B(\HH)$ the operator defined by
$e_{ij}e_k=\delta_{jk}e_i$ and define the ``bar'' of an operator by
$\bar{e_{ij}}=e_{ij}$, extended antilinearly to $\B(\HH)$.
Thus $a_k=\sum_{i,j}c^{ij}_ke_{ij}$, for some $c^{ij}_k\in\C$, and
$U=\sum_{i,j}e_{i,j}\ot\sum_kc^{ij}_kq_k$. By definition, we have
$$
\bar{U}=
\sum_{i,j}e_{ij}\ot\left( \sum\nolimits_kc^{ij}_kq_k \right)^*=
\sum_{i,j}e_{ij}\ot\left( \sum\nolimits_k\overline{c^{ij}_k}q_k^* \right)=
\sum_k\overline{\sum\nolimits_{ij}c^{ij}_ke_{ij}}\otimes q_k^*
=\sum_k\bar{a_k}\ot q^*_k 
\;.
$$
This completes the  proof of the lemma.
\end{proof}

\begin{prop}\label{UUbarV}
Let $(Q,U)\in\mathbf{C_{u,\R}}(\A,\mathbb{I})$.
Then\vspace{-3pt}
\begin{list}{}{\itemsep=0pt \leftmargin=1.2em}
\item[1.] $U^\pi:=(\pi_\C\otimes\id)(U)$ is a unitary corepresentation of $Q$ on the Hilbert space $\HH$;
\item[2.] $U^{\bar\pi}:=(j\otimes *)(U^\pi)$ is a unitary corepresentation of $Q$ on $\HH$, where $ j ( a ) = J a J^* \;\forall\;a\in\B(\HH)$;
\item[3.] $V=U^\pi U^{\bar\pi}$ is a unitary corepresentation of $Q$ on the Hilbert space $\HH$. Note that: %
\begin{equation}\label{eq:Vexp}
V=\sum\nolimits_{i,j}U^\pi_{(1)i}JU^\pi_{(1)j}J^{-1}\otimes U^\pi_{(2)i}(U^\pi_{(2)j})^* \;.
\end{equation}
\end{list}
\end{prop}

\begin{proof}
Unitarity of $U^\pi$ follows from the fact that $\pi_\C$ is a unital $*$-representation and
$U$ is a  unitary. Furthermore,
\begin{align*}
(\id\otimes\Delta)(U^\pi) &=(\id\otimes\Delta)(\pi_\C\otimes\id)(U)
=(\pi_\C\otimes\id)(\id\otimes\Delta)(U)
\\[3pt]
&=(\pi_\C\otimes\id)(U_{(12)}U_{(13)})
=(\pi_\C\otimes\id)(U_{(12)})(\pi_\C\otimes\id)(U_{(13)})
=U^\pi_{(12)}U^\pi_{(13)} \;.
\end{align*}
This proves that $U^\pi$ is a corepresentation.

\smallskip

\noindent
To prove 2, we compute
\begin{align*}
(\id\otimes\Delta)(U^{\bar\pi}) &=(\id\otimes\Delta)(\jmath\otimes *)(U^\pi)
=(\jmath\otimes *\otimes *)(\id\otimes\Delta)(U^\pi)
\\[3pt]
&=(\jmath\otimes *\otimes *)(U^\pi_{(12)}U^\pi_{(13)})
=(\jmath\otimes *\otimes *)(U^\pi_{(12)})(\jmath\otimes *\otimes *)(U^\pi_{(13)})
=U^{\bar\pi}_{(12)}U^{\bar\pi}_{(13)} \;.
\end{align*}
proving that $U^{\bar\pi}$ is a corepresentation.

To prove unitarity, let us fix a basis of $\HH$ and denote by $J_0$ the unitary operator
obtained by composing $J$ with the componentwise conjugation in this basis, and by $\bar{a}$
as usual the ``bar'' of an $a\in\B(\HH)$ considered as a matrix in the chosen basis. Since, $J^2=\epsilon 1,$ $\epsilon=\pm 1$, so $J^*=\epsilon J$ and we have
\begin{equation*}
\begin{split}
JaJ^*(v)&=Ja\epsilon J(v)=\epsilon Ja(J_0 \overline{v})\\[2pt]
&=\epsilon J_0(\overline{a(J_0(\overline{v}))})=\epsilon J_0 (\overline{a}\overline{J_0(\overline{v})})=\epsilon J_0\overline{a}\overline{J_0}v, 
\end{split}
\end{equation*}
where $a\in B(\HH),~v\in\HH,$ which proves that for $a\in B(\HH),$ we have 
\begin{equation}\label{eq:number}
JaJ^{-1}=\epsilon J_0\bar{a}\bar{J_0} \;.
\end{equation}
Using this, it is easy to see that $U^{\bar\pi}= \epsilon (J_0\otimes 1)\bar{ U^{\pi}} (\bar{J_0}\otimes 1).$ Since $ U $ is a biunitary, $ U^\pi $ is a unitary implies that $\bar{U^{\pi}}$ is unitary. Thus, $U^{\bar\pi} $ is a product of three unitary operators and hence is a unitary.

\smallskip

\noindent
Now we prove 3. $V$ is a product of two unitary operators and hence is a unitary.  Moreover,
$$
(\id\otimes\Delta)(V)=(\id\otimes\Delta)(U^\pi)(\id\otimes\Delta)(U^{\bar\pi})=
U^\pi_{(12)}U^\pi_{(13)}U^{\bar\pi}_{(12)}U^{\bar\pi}_{(13)}
 \;.
$$
We notice that due to points 2 and 3 of Lemma \ref{newlemmaforgaugegroupputbyme}, 
and also equation \eqref{eq:number}, $U^\pi_{(13)} $ commutes with $ U^{\bar\pi}_{(12)}.$
This proves that $V$ is a corepresentation.
\end{proof}

\begin{df}
The CQG generated by the matrix coefficients of the unitary corepresentation \eqref{eq:Vexp}, when $(Q,U)=(Q_{u,\R}(\A),U_0)$ is the
quantum unitary group of $\A$, will be called \emph{quantum gauge group}
of the finite spectral triple $(\A,\HH,D,\gamma,J)$, and will
be denoted by $\hat{\mc{G}}(\A,J)$.
\end{df}

\begin{rem}
Using \eqref{eq:number} in \eqref{eq:Vexp},
we can rewrite the latter %
in the following equivalent way:
$$
V=\epsilon\,U^{\pi} ( J_0 \otimes id ) \bar{ U^{\pi} } ( \bar{J_0} \otimes
id )
$$
We are going to use this equation
in the next three subsections, where we compute the quantum
gauge group for the Einstein-Yang-Mills system, the spectral
triple on $\A^{\mathrm{ev}}=\Q\oplus\Q\oplus M_4(\C)$, and
the finite noncommutative space of the Standard Model.
Note that in the three above-mentioned examples we have
$\epsilon=1$.
\end{rem}

\subsection{The Einstein Yang-Mills system}
In this section, we consider the following five families of real spectral triples:
\begin{itemize}
\item[i)] $\A=M_n(\C)$, $\HH=M_n(\C)$, $D=0$, $J(a)=a^*$;
\item[ii)] $\A=M_n(\R)$, $\HH=M_n(\C)$, $D=0$, $J(a)=a^*$;
\item[iii)] $\A=M_n(\Q)$, $\HH=M_{2n}(\C)$, $D=0$, $J(a)=a^*$;
\item[iv)] $\A=M_n(\C)$, $\HH=M_n(\C)\oplus M_n(\C)$, $D=0$, $J(a\oplus b)=a^*\oplus b^*$;
\item[v)] $\A=M_n(\C)$, $\HH=M_n(\C)\oplus M_n(\C)$, $D=0$, $J(a\oplus b)=b^*\oplus a^*$.
\end{itemize}
In the first case $\HH=\A$ and we think of $\A$ as a complex algebra,
while in the last four $\HH=\A_{\C}$ and we think of $M_n(\C)$ in (iv) and (v) as a
real algebra (note that the representation is not complex linear in these cases).
In all five cases the inner product is the
Hilbert-Schmidt inner product $\inner{a,b}_{HS}:=\tr(a^*b)$,
the representation $\pi$ is the restriction to $\A$ of the GNS
representation of $\A_\C$ (resp.~the GNS representation of $\A$ in
the first case) associated to the trace. Note that in the
cases (iv) and (v), the representation is
\begin{equation}\label{eq:piofa}
\pi(a)(b\oplus c)=ab\oplus \bar{a}c \;,\qquad \forall\;a,b,c\in M_n(\C),
\end{equation}
since we identify $\A=M_n(\C)$ with the real subalgebra of
$\A_{\C}=M_n(\C)\oplus M_n(\C)$ of elements of the form $a\oplus\bar a$.
From now on, the representation symbol will be omitted.

\begin{lemma}
In the five cases above, the gauge group is
$\mathcal{G}(\A,J)=PU(n)$, $PO(n)$, $PSp(n)$, $PU(n)$, and $U(n)/\Z_2$, respectively.
\end{lemma}
\begin{proof}
The classical gauge group $\mathcal{G}(\A,J)$ is the quotient of $U(\A)$
by the kernel of the adjoint representation $u\mapsto uJuJ^{-1}$ on $\HH$.
In the cases (i)-(iii): since $uJuJ^{-1}a=uau^*$ for all $u\in U(\A)$ and $a\in\HH$,
$uJuJ^{-1}=1$ if and only if $uau^*=a$ for all $a\in\HH$,
that is $ua=au$; this implies $u=\lambda 1$ with $\lambda\in\C$,
since $M_k(\C)$ has trivial center. In case (iv),
\begin{equation}\label{eq:numberfour}
uJuJ^{-1}(a\oplus b)=uau^*\oplus \bar{u}bu^t
\end{equation}
for all $u\in U(\A)$ and $a\oplus b\in\HH$,
and $uJuJ^{-1}=1$ if and only if $uau^*=a$ for all $a\in M_n(\C)$
(the condition $\bar{u}bu^t=b$ is equivalent to
$uau^*=a$ with $a=\bar{b}$), so that one
reaches the same conclusion. Hence $\mathcal{G}(\A,J)=U(\A)/\{U(\A)\cap U(1)\}=PU(\A)$ in the cases
(i)-(iv).

The case (v) is similar to case (iv), but instead of \eqref{eq:numberfour} one gets the
condition
$$
uJuJ^{-1}(a\oplus b)=uau^t\oplus \bar{u}bu^*
$$
for all $a,b\in M_n(\C)$ and $u\in U(\A)=U(n)$.
The kernel of the adjoint representation is given by elements $u\in U(n)$
such that $ua=a\bar{u}$ for all $a\in M_n(\C)$.
For $a=\mathbb{I}_n$ we
get $u=\bar{u}$, that is $u\in O(n)$. The kernel is then the
set of $u\in O(n)$ such that $ua=au$ for all $a\in M_n(\C)$.
Since the center of $M_n(\C)$ is trivial, we find $u=\lambda 1$
with $\lambda\in\R$. Unitarity gives $\lambda=\pm 1$.
This proves that in case (v), $\mathcal{G}(\A,J)=U(n)/\Z_2$.
\end{proof}

Let us explain the physical interest for the spectral triples above.

The spectral triple in (i) is the finite part of the spectral triple
studied for example in Sec.~11.4 of \cite{CM08}, describing the geometry
of a (Euclidean) $SU(n)$ Yang-Mills theory minimally coupled to gravity.
We remark that in Connes' approach, gauge fields are connections with
coefficients in the Lie algebra $\mathfrak{g}$ of the gauge group, and
since in (i) the gauge group is $PU(n)=SU(n)/\Gamma_n$, that has the same
Lie algebra as $SU(n)$, one speaks about $SU(n)$ gauge theory. Here $\Gamma_n\simeq\Z_n$
is the group of $n$-th roots of unity.

Similarly, since $PO(n)$ has the same Lie algebra of $SO(n)$ and $PSp(n)$
has the same Lie algebra of $Sp(n)$, with the spectral triples in (ii) and
(iii) one can construct $SO(n)$ and $Sp(n)$ Yang-Mills theories, respectively,
as one can see adapting the proof of \cite[Prop.~1.157]{CM08}.

In the case (iv), $\HH=M_n(\C)\oplus M_n(\C)$ is doubled with respect
to case (i), and while the first summand transforms according to the
representation $(u,a)\mapsto uau^*$ of $U(n)$, the second summand transforms
according to the dual representation $(u,b)\mapsto \bar{u}bu^t$
(cf.~equation \eqref{eq:numberfour}), i.e.~like a pair particle-antiparticle.
One gets then a $SU(n)$ Yang-Mills
theory but with a sort of fermion-doubling phenomenon (the particle-antiparticle
distinction is already present in the continuous part of the full spectral triple).

Finally, as explained in the proof of \cite[Prop.~1.157]{CM08},
it is not possible to get quantum electrodynamics
(a $U(1)$ gauge theory) from the example (i), because for $n=1$
the adjoint action has the whole group $U(\A)=U(1)$ in the kernel.
The solution used in \cite{walter} to get a $U(1)$ gauge theory
is to use a two-point space, i.e.~to ``double'' the spectral
triple in (i), for $n=1$. The spectral triple $(\A_{\C},\HH,D,J)=(\C^2,\C^2,0,J)$ considered
in Sec.~3.3 of \cite{walter} is the ``complexification'' ---
in the sense we discussed in Sec.~\ref{sec:catCJ}, cf.~equation \eqref{eq:piC} --- of
our spectral triple
(v) for $n=1$.
In \cite{walter}, the authors use the gauge group of the complexified
spectral triple and prove that it is $U(1)$, but we remark here that
using the real spectral triple (v) one reaches the same result: for $n=1$,
$\mathcal{G}(\A,J)=U(1)/\Z_2\simeq U(1)$. In general, for arbitrary $n$,
since $U(n)/\Z_2$ has the same Lie algebra as $U(n)$, by applying the
spectral action machinery one gets a $U(n)$ Yang-Mills theory (minimally
coupled to gravity).

\smallskip

We now compute the quantum gauge group of the spectral triples above.
The computation is completely analogous to the one of the classical gauge group.
We remark that a real structure similar to that of (v)
will be used for the spectral triple on $\A^{\mathrm{ev}}$,
discussed in the next Sec.~\ref{sec:aev}. 

\begin{prop}\label{prop:qgg}
In the cases (i)-(iii), the quantum gauge group $\hat{\mathcal{G}}(\A,J)$
is the projective version of $Q_{u,\R}(\A)$;
thus $\hat{\mathcal{G}}(\A,J)=PA_u(n),PA_o(n),PA_{sp}(n)$, respectively.
In case (iv), $\hat{\mathcal{G}}(\A,J)$ is generated by products $u_{lm}u_{kj}^*$
and $u_{li}^*u_{kj}$, where $u_{ij}$ are the canonical generators of $A_u(n)$.
In case (v), $\hat{\mathcal{G}}(\A,J)$ is generated by products $u_{lm}u_{kj}$.
Note that both in case (iv) and (v), $PA_u(n)$ is a $C^*$-subalgebra of $\hat{\mathcal{G}}(\A,J)$.
\end{prop}

\begin{proof}
Let us start with case (i): $\A=M_n(\C)$, $\HH=M_n(\C)$ and $J(a)=a^*$, that is $J_0(a)=a^t$.
Let $U=(u_{ij})$ be the fundamental corepresentation of $A_u(n)$, $e_{ij}$ be the canonical
basis of $M_n(\C)$ and $\pi$ the representation of $\A$.
Note that modulo the identification of $\pi(\A)$ with $\A$,
the corepresentation $U^\pi$ in Prop.~\ref{UUbarV} is $U^\pi=\sum_{ij}e_{ij}\otimes u_{ij}$.
Therefore:
\begin{align*}
V(e_{ij}) &=U^\pi (J_0\otimes\id)\bar{U^\pi}(J_0\otimes\id)(e_{ij}\otimes 1)
\\
&=U^\pi (J_0\otimes\id)\bar{U^\pi}(e_{ji}\otimes 1)
=U^\pi (J_0\otimes\id)\sum\nolimits_{kl}e_{kl}e_{ji}\otimes u_{kl}^*
\\
&=U^\pi (J_0\otimes\id)\sum\nolimits_{k}e_{ki}\otimes u_{kj}^*
=U^\pi \sum\nolimits_{k}e_{ik}\otimes u_{kj}^*
\\
&=\sum\nolimits_{klm}e_{lm}e_{ik}\otimes u_{lm}u_{kj}^*
=\sum\nolimits_{kl}e_{lk}\otimes u_{li}u_{kj}^* \;.
\end{align*}
From this, it follows that $\hat{\mathcal{G}}(\A,J)$ is generated by
elements $u_{li}u^*_{kj}$, i.e.~it is $PQ_{u,\R}(\A)=PA_u(n)$.

In the cases (ii) and (iii) the proof is exactly the same, except that one gets
$PQ_{u,\R}(\A)=PA_o(n)$ and $PQ_{u,\R}(\A)=PA_{sp}(n)$, respectively.

In the case (iv), if $U=(u_{ij})$ is the fundamental corepresentation of $Q_{u,\R}(M_n(\C))=A_u(n)$,
from the proof of Prop.~\ref{prop:FKH} with $F=\mathbb{I}$ we see that $U^\pi=U\oplus\bar{U}$,
or explicitly
$$
U^\pi(e_{ij,r}\otimes 1)=\begin{cases}
\sum_ke_{kj,r}\otimes u_{ki} &\mathrm{if}\;r=1 \,,\\
\sum_ke_{kj,r}\otimes u_{ki}^* &\mathrm{if}\;r=2 \,,
\end{cases}
$$
where $e_{ij,r}$ and $e_{ij,2}$ are the canonical bases of the two copies of $M_n(\C)$ in $\HH$.
A computation similar to the one for cases (i)-(iii), but with $J_0(e_{ij,r})=e_{ji,r}$, gives
\begin{equation}\label{eq:change}
V(e_{ij,r})=\begin{cases}
\sum\nolimits_{kl}e_{lk,r}\otimes u_{li}u_{kj}^* &\mathrm{if}\;r=1 \,,\\
\sum\nolimits_{kl}e_{lk,r}\otimes u_{li}^*u_{kj} &\mathrm{if}\;r=2 \,.
\end{cases}
\end{equation}
Thus $\hat{\mathcal{G}}(\A,J)$ is generated by products $u_{lm}u_{kj}^*$
and $u_{li}^*u_{kj}$.

The case (v) is similar, except for $J_0(e_{ij,1})=e_{ji,2}$ and $J_0(e_{ij,2})=e_{ji,1}$.
The formula for $U^\pi$ is the same, but now due to a different real
structure, instead of \eqref{eq:change} we find:
$$
V(e_{ij,r})=\begin{cases}
\sum\nolimits_{kl}e_{lk,r}\otimes u_{li}u_{kj} &\mathrm{if}\;r=1 \,,\\
\sum\nolimits_{kl}e_{lk,r}\otimes u^*_{li}u_{kj}^* &\mathrm{if}\;r=2 \,.
\end{cases}
$$
Thus, $\hat{\mathcal{G}}(\A,J)$ is generated by products $u_{li}u_{kj}$
and their adjoints. This concludes the proof.
\end{proof}

\subsection{The spectral triple on $\A^{\mathrm{ev}}$}\label{sec:aev}

The spectral triple discussed in this section can be found in \cite{whySM}.
The data is the following. The Hilbert space is $M_4(\C)\oplus M_4(\C)$, with inner product
$\inner{a\oplus b,c\oplus d}=\tr(a^*c+b^*d)$. 
The real structure is the map $a\oplus b\mapsto b^*\oplus a^*$,
where $*$ is the Hermitian conjugation.
The algebra $\A^{\mathrm{ev}}=\Q\oplus\Q\oplus M_4(\C)$
acts on $M_4(\C)\oplus M_4(\C)$ by left multiplication.
Here we identify $\Q$ with the real subalgebra of
$M_2(\C)$ with elements
\begin{equation}\label{eq:XXX}
q=\begin{mat}
\alpha & \beta \\
-\bar\beta & \bar\alpha
\end{mat}
\end{equation}
for $\alpha,\beta\in\C$,
and we identify accordingly $\Q \oplus \Q $ with the corresponding real
subalgebra of $M_2(\C) \oplus M_2 ( \C ) \subset M_4 ( \C )$.

For computational reasons, it is useful to rewrite the spectral triple as follows.
The map $e_{ij}\oplus 0\mapsto e_i\otimes e_j\otimes e_1$ and
$0\oplus e_{ij}\mapsto e_i\otimes e_j\otimes e_2$
(with $e_i$ canonical orthonormal basis vectors)
is an isometry between $M_4(\C)\oplus M_4(\C)$ and the
Hilbert space $\HH:=\C^4\otimes\C^4\otimes\C^2$;
the corresponding representation $\pi$ of $\A^{{\rm ev}}$ on $\HH$
is given by
$$
\pi(a,b)=a\otimes \mathbb{I}_4\otimes e_{11}+b\otimes \mathbb{I}_4\otimes e_{22} ,
$$
where $a\in \Q \oplus \Q $ and $b\in M_4(\C)$.

The complex linear span of $\pi(\A^{{\rm ev}})$ inside $\B(\HH)$ is the complex $C^*$-algebra
$M_2(\C) \oplus M_2 ( \C ) \oplus M_4(\C)$. The real structure becomes the antilinar operator
$J$ given by
$$
J(\xi\otimes\zeta\otimes\eta)=
\zeta^*\otimes\xi^*\otimes\begin{pmatrix} 0 & 1 \\ 1 & 0\end{pmatrix}\eta^* \;,
$$
where now $*$ is the componentwise conjugation on $\C^4$ resp.~$\C^2$.

Since we are going to need the CQG $ Q_{u,\R} ( \A^{{\rm ev}}   ) \cong C(SU(2))\ast C(SU(2))\ast A_u(4), $ we fix the notation for its generators.
The symbols $ ( a_{ij} )_{i,j = 1,2}$ and $ ( a_{kl} )_{k,l = 3,4}$ denote the canonical generators of the first and second copy of $C(SU(2))$.
That is, the first resp.~the second copy of $C(SU(2))$ is generated by the matrix elements of a biunitary matrix
$$
\bigg(\!\begin{array}{cc}
a_{11} & a_{12} \\ -a_{12}^* & a_{11}^*
\end{array}\!\bigg)\;,\qquad\mathrm{resp.}\quad
\bigg(\!\begin{array}{cc}
a_{33} & a_{34} \\ -a_{34}^* & a_{33}^*
\end{array}\!\bigg)
\;,
$$
and $a_{21}=-a_{12}^*$, $a_{22}=a_{11}^*$, $a_{43}=-a_{34}^*$, $a_{44}=a_{33}^*$, gives the involution.
For the generators of $A_u(4),$ we use the usual symbols $(u_{ij})_{i,j=1}^4.$
We will denote the canonical basis of 
$ \Q_{\C} \cong M_2 ( \C )  $ by the symbols $ F_{ij}, ~ i,j = 1,2  $ while $E_{ij,k}$, for $i,j=1,\ldots,4$ and $k=1,2$, will denote the
generators of the $k$-th copy of $M_4(\C)$ in $M_4(\C)_\C:=M_4(\C)\oplus M_4(\C)$.
The unitary corepresentations of $ Q_{u,\R} ( \Q ) \cong C(SU(2))$ on $ L^2 ( M_2 ( \C ), \tr ) $ and of $ Q_{u,\R} ( M_4(\C) ) \cong A_u(4)$
on $L^2 ( M_4 ( \C ) \oplus M_4 ( \C ), \tr   ) $, respectively, are given by 
\begin{align*}
W_{\Q} &:=F_{11}\otimes a_{11} +F_{12}\otimes a_{12} +F_{21}\otimes a_{21} ~+F_{22}\otimes a_{22} \;,\\
W_{M_4 ( \C )} &:=\sum_{i,j=1}^4 E_{ij,1}\otimes u_{ij}+\sum_{i,j=1}^4 E_{ij,2}\otimes u^*_{ij} \;.
\end{align*}

\begin{lemma}\label{lemma:m4c}
We have:
\begin{equation}\label{quaternion}
\pi_\C(F_{ij})=e_{ij}\otimes 1 \otimes e_{11},~i,j=1,2 \;,
\end{equation}
and
\begin{subequations}\label{m4c}
\begin{align}
(\pi_\C\otimes\id) W_{M_4 ( \C )} &=\sum_{i,j=1}^4 e_{ij}\otimes1\otimes e_{22}\otimes u_{ij},\\
(\pi_\C\otimes\id)\bar{W_{M_4 ( \C )}} &=\sum_{i,j=1}^4 e_{ij}\otimes1\otimes e_{22}\otimes u^*_{ij}.
\end{align}
\end{subequations}
\end{lemma}

\begin{proof}
The isomorphism between the complex $C^*$-algebras $M_2(\C)$ and $\Q\otimes_\R\C$ is  determined by $m\rightarrow m\otimes_\R1$ and $i m\rightarrow m\otimes_\R i$ for $m\in\Q.$ The equations \ref{quaternion} can now be derived easily by observing that:
\begin{align*}
F_{11} &=i\bigg(\!\begin{array}{cc}
                     -i\frac{1}{2}&0\\0&i\frac{1}{2}
                     \end{array}\!\bigg)+
\bigg(\!\begin{array}{cc}
\frac{1}{2}&0\\0&\frac{1}{2}\end{array}\!\bigg) \;, &
F_{12} &=i\bigg(\!\begin{array}{cc}
                     0&-i\frac{1}{2}\\-i\frac{1}{2}&0
                     \end{array}\!\bigg)+
\bigg(\!\begin{array}{cc}
0&\frac{1}{2}\\-\frac{1}{2}&0
\end{array}\!\bigg)\;, \\
F_{21}&=i\bigg(\!\begin{array}{cc}
                     0&-i\frac{1}{2}\\-i\frac{1}{2}&0
                     \end{array}\!\bigg)+
\bigg(\!\begin{array}{cc}
0&-\frac{1}{2}\\\frac{1}{2}&0
\end{array}\!\bigg) \;, &
F_{22}&=i\bigg(\!\begin{array}{cc}
                     i\frac{1}{2}&0\\0&-i\frac{1}{2}
                     \end{array}\!\bigg)+
\bigg(\!\begin{array}{cc}
\frac{1}{2}&0\\0&\frac{1}{2}
\end{array}\!\bigg) \;.
\end{align*}
Similarly, the
isomorphism between $M_4(\C)\oplus M_4(\C)$ and $M_4(\C)\otimes_\R\C,$ is determined by $(a\oplus\bar{a})\rightarrow a\otimes_\R 1$ and $i(a\oplus\bar{a})\rightarrow a\otimes_\R i.$ Under this isomorphism, $a\oplus0$ and $0\oplus\bar{a}$ get mapped to $\frac{1}{2}(a\otimes_\R1-ia\otimes_\R i)$ and  $\frac{1}{2}(a\otimes1+ia\otimes i)$, respectively. Using these facts, \eqref{m4c} follows easily.
\end{proof}

We now compute the quantum gauge group.

\begin{prop}\label{prop:4.7}
The quantum gauge group $\hat{\mc{G}}(\A^{{\rm ev}},J)$ of the spectral triple above is the projective version of $ C(SU(2))\ast C(SU(2))\ast A_u(4)$.
\end{prop}

\begin{proof}
Let $ U_1 = \sum_{i,j=1}^2 e_{ij}\otimes1\otimes e_{11}\otimes a_{ij} +\sum_{i,j=3}^4 e_{ij}\otimes1\otimes e_{11}\otimes a_{ij},~ U_2 = \sum_{i,j=1}^4 e_{ij}\otimes1\otimes e_{22}\otimes u_{ij} $ and $ V $ be the unitary corepresentation of $ C(SU(2))\ast C(SU(2))\ast A_u(4) $ on $ L^2 ( (\A^{{\rm ev}})_\C, \tr    ). $
Then, by using \eqref{quaternion} and \eqref{m4c}, we have,
\begin{align*}
( \pi_{\C} \otimes \id ) V &=U_1\oplus U_2 \;.
\end{align*}
In order to calculate
$(\pi_\C\otimes\id)V(J_0\otimes\id)(\pi_\C\otimes\id)\bar{V}(J_0\otimes\id),$ we first observe that  $U_k(J_0\otimes\id)\bar{U_k}(J_0\otimes\id)=0,$ for $k=1,2.$ Thus the only contributing terms are $U_1(J_0\otimes\id)\bar{U_2}(J_0\otimes\id)$ and $U_2(J_0\otimes\id)\bar{U_1}(J_0\otimes\id).$

A direct computation, using equations \eqref{m4c} and \eqref{quaternion}, yields:
\begin{align*}
U_1(J_0\otimes\id)\bar{U_2}(J_0\otimes\id)&=
\sum_{k,l\in\{1,2\}~\mbox{or}~\{3,4\}}\sum_{i,j=1}^4 e_{kl}\otimes e_{ij}\otimes e_{11}\otimes a_{kl}u^*_{ij} \;;\\
U_2(J_0\otimes\id)\bar{U_1}(J_0\otimes\id)&=
\sum_{k,l\in\{1,2\}~\mbox{or}~\{3,4\}}\sum_{i,j=1}^4 e_{ij}\otimes e_{kl}\otimes e_{22}\otimes u_{ij}a^*_{kl} \;.
\end{align*}
Hence we have,
\begin{multline*}
\qquad
(\pi_\C\otimes\id)V(J_0\otimes\id)(\pi_\C\otimes\id)\bar{V}(J_0\otimes\id)\\
=\sum_{k,l\in\{1,2\}~\mbox{or}~\{3,4\}}\sum_{i,j=1}^4 e_{kl}\otimes e_{ij}\otimes e_{11}\otimes a_{kl}u^*_{ij}\\
+\sum_{k,l\in\{1,2\}~\mbox{or}~\{3,4\}}\sum_{i,j=1}^4 e_{ij}\otimes e_{kl}\otimes e_{22}\otimes u_{ij}a^*_{kl}.
\qquad
\end{multline*}
Thus the quantum gauge group of the spectral triple on $ \A^{{\rm ev}} $ is  generated by the elements $\{a_{kl}u^*_{ij}:~k,l=1,2~or~3,4,~i,j=1,,,4\}.$

The proof will be completed, if we show that elements of the form $a_{kl}a^*_{ij}$ for $i,j,k,l=1,2~or~3,4$ and $u_{ij}u^*_{kl}$ for $i,j,k,l=1,\ldots,4,$ belong to this CQG.
This we show as follows.

We note that since $(u_{ij})_{i,j=1}^4$ generate $A_u(4),$ $\sum_{j=1}^4 u_{ij}^*u_{ij}=1.$ Thus we have $a_{ij}a^*_{kl}=\sum_{m=1}^4 a_{ij}u^*_{pm}u_{pm}a^*_{kl},$ which proves that $a_{ij}a^*_{kl}$ belongs to the CQG.

Moreover, since the each of the sets $\{a_{11},a_{12}\}$ and $\{a_{33},a_{34}\}$ are set of generators of  $C(SU(2)),$ we have $a_{11}a_{11}^*+a_{12}^*a_{12}=1=a_{33}a^*_{33}+a_{34}a_{34}^*.$ Hence we have $u_{ij}u^*_{kl}=(u_{ij}a_{11})(a^*_{11}u_{kl})+(u_{ij}a_{12})(a^*_{12}u_{kl}),$ which proves that $u_{ij}u^*_{kl}$ belongs to the CQG. This proves the result.
\end{proof}

\subsection{The finite-dimensional spectral triple of the Standard Model}

For the spectral triple $(\A_F,H_F,D_F,\gamma_F,J_F)$
describing the internal space $F$ of the Standard Model
(cf.~\cite{CM08} and references therein)
we will use the notations of \cite{ludwik}.
In particular, the algebra $\A_F$ and the Hilbert space $H_F$ are given by
\begin{align*}
\A_F &=\C\oplus\Q\oplus M_3(\C) \;, & \hspace{-1cm} 
H_F  &=\C^2\otimes\C^4\otimes\C^4\otimes\C^3 \;.
\end{align*}
The real structure $J_F$ is the composition of the componentwise
complex conjugation on $H_F$ with the linear operator
$$
( J_F )_{0}:=
1\otimes 1\otimes {\footnotesize
\begin{pmatrix}
0 & 0 & 1 & 0 \\
0 & 0 & 0 & 1 \\
1 & 0 & 0 & 0 \\
0 & 1 & 0 & 0
\end{pmatrix}}
\otimes 1 \;,
$$
and an element $a=(\lambda,q,m)\in \A_F$ (with $\lambda\in\C$, $q\in\Q$ and $m\in M_3(\C)$) is represented by
\begin{align*}
\pi^F(a) &=
q\otimes 1\otimes e_{11}\otimes 1 +
{\small\begin{mat}
\lambda & 0 \\ 0 & \bar\lambda
\end{mat}}\otimes 1
\otimes e_{44}\otimes 1 \\[2pt]  & \qquad +
1\otimes
{\small
\begin{pmatrix}
\lambda & 0 & 0 & 0 \\
0 \\
0 && {\normalsize m} \\
0
\end{pmatrix}}
\otimes (e_{22}+e_{33})\otimes 1 \;,
\end{align*}
where $m$ is a $3\times 3$ block.
We will not need the grading and the Dirac operator.

\begin{prop}\label{quantum gauge group of the Standard Model}
The quantum gauge group $\hat{\mc{G}}(\A_F,J_F)$
of the finite geometry of the Standard Model is the projective version of $C(U(1))\ast C(SU(2))\ast A_u(3).$ 
\end{prop}

\begin{proof}
The proof is analogous to the one of Lemma \ref{lemma:m4c} and Prop.~\ref{prop:4.7}.

Let $W_{\C}$, $W_{\Q} $ and $W_{M_3 ( \C ) }$ be the unitary corepresentations of  $Q_{u,\R}(\C)\simeq C(U(1))$, $Q_{u,\R}(\Q)\simeq C(SU(2))$
and $Q_{u,\R}(M_3(\C))\simeq A_u(3)$ on $L^2(\C\oplus\C,\tr)$, $L^2(M_2(\C),\tr)$ and $L^2(M_3(\C)\oplus M_3(\C),\tr)$, respectively.
Proceeding as in the above-mentioned lemma, we can deduce that:
\begin{subequations}
\begin{align}
(\pi^{F}_\C\ot\id)W_{\C}
&=\{\mathbb{I}_2\ot \mathbb{I}_4\ot e_{44}\ot \mathbb{I}_3+\mathbb{I}_2\ot e_{11}\ot (e_{22}+e_{33})\ot \mathbb{I}_3\}\ot z ;
\label{1st term of Standard Model} \\
(\pi^F_\C\ot\id)W_{\Q} &= \sum_{i,j=1}^2 e_{ij}\ot \mathbb{I}_4\ot e_{11}\ot \mathbb{I}_3\ot a_{ij},  
\label{2nd term of Standard Model} \\
(\pi^F_\C\ot\id)W_{M_3 ( \C ) } &=\sum_{i,j=2}^4 \mathbb{I}_2\ot e_{ij}\ot (e_{22}+e_{33})\ot \mathbb{I}_3\ot u_{ij}, 
\label{3rd term of Standard Model}
\end{align}
\end{subequations}
where $\{a_{11},a_{12}\}$ are the generators of $C(SU(2))$, as in previous section,
$z$ is the unitary generator of $C(U(1))$
and $\{u_{ij}\}_{i,j=2}^4$ generate $A_u(3)$ .

Let $V=W_{\C} \oplus W_{\Q} \oplus W_{M_3 ( \C ) }.$
Then, using the equations \eqref{1st term of Standard Model},\eqref{2nd term of Standard Model} and \eqref{3rd term of Standard Model}, we have
\begin{align}
&\hspace{-1.5cm}(\pi_\C\ot\id)(V)((J_F)_0\ot\id)(\pi_\C\ot\id)(\bar{V})((J_F)_0\ot\id) \notag\\[8pt]
&\hspace{-10pt}
=\{\mathbb{I}_2\ot \mathbb{I}_4\ot e_{44}\ot \mathbb{I}_3+\mathbb{I}_2\ot e_{11}\ot e_{44}\ot \mathbb{I}_3\}\ot z\bar{z} \notag \\
& +\sum_{i,j=1}^2 e_{ij}\ot e_{11}\ot e_{33}\ot \mathbb{I}_3\ot za^*_{ij}+\sum_{i,j=2}^4 \mathbb{I}_2\ot e_{ij}\ot e_{44}\ot \mathbb{I}_3\ot zu^*_{ij} \notag\\
& +\sum_{i,j=1}^2 e_{ij}\ot e_{11}\ot e_{11}\ot \mathbb{I}_3\ot a_{ij}\bar{z}+\sum_{i,j=2}^4\sum_{k,l=1}^2 e_{kl}\ot e_{ij}\ot e_{11}\ot \mathbb{I}_3\ot a_{kl}u^*_{ij} \notag\\
& +\sum_{i,j=2}^4 \mathbb{I}_2\ot e_{ij}\ot e_{22}\ot \mathbb{I}_3\ot u^*_{ij}\bar{z}+\sum_{i,j=2}^4\sum_{k,l=1}^2 e_{kl}\ot e_{ij}\ot e_{33}\ot \mathbb{I}_3\ot u_{ij}a^*_{kl}.
\label{sango holo vober khela}
\end{align}
From equation \eqref{sango holo vober khela}, the result follows by arguing as in Proposition \ref{prop:4.7}.
\end{proof}

\section{Quantum automorphisms of real $C^*$-algebras}\label{sec:5}

The following observations is a simple restatement of Lemma 5.1 and 5.2 of \cite{ludwik}.

\begin{lemma}\label{lemma:qautreal}
Let $G$ be the group of automorphism of a real $C^*$-algebra $\A$. Then $C(G)$ is
the universal object in the category of commutative CQGs $Q$ with a coaction
$\alpha:\A_{\C}\to\A_{\C}\otimes Q$ such that
$(\sigma\otimes \ast_{C ( G ) })\circ\alpha=\alpha\circ\sigma$.
\end{lemma}

In Sec.~\ref{sec:5.1}, we define the category of quantum automorphisms of a 
finite-dimensional real $C^*$-algebra and prove the existence of a universal object.
In Sec.~\ref{sec:5.2}, we compute the universal object for $M_n(\R)$ and $M_n(\Q)$.
In Sec.~\ref{sec:5.3}, we discuss quantum automorphisms and isometries of $M_n(\C)$, thought
of as a real algebra.

\subsection{Definition and existence of the quantum automorphism group}\label{sec:5.1}

Motivated by Lemma \ref{lemma:qautreal}, we define quantum automorphisms of a finite-dimensional real\linebreak $C^*$-algebra
as follows.

\begin{df}\label{def:5.2}
Let $\A$ be a finite-dimensional real $C^*$-algebra.
We denote by $\qaut(\A)$ the category whose objects $(Q,\alpha)$ are pairs, with $Q$ a CQG
and $\alpha:\A_{\C}\to\A_{\C}\otimes Q$ a coaction on $\A_{\C}$ preserving the trace and such that
\begin{equation}\label{eq:alphasigma}
(\sigma\otimes \ast)\circ\alpha=\alpha\circ\sigma \;,
\end{equation}
and morphisms $(Q,\alpha)\to (Q',\alpha')$ are CQGs homomorphisms $\phi:Q\to Q'$ intertwining
the coactions, i.e.~$(\id\otimes\phi)\alpha=\alpha'$.
The universal object in this category, if exists, will be denoted $A_{aut,\R}(\A)$ and called
the quantum automorphism group of $\A$.
\end{df}

Notice that the coinvariance of the trace is automatic in the commutative case, while in general it must
be imposed as an additional condition (necessary to prove the existence of the universal object).
Now we prove the existence of the quantum automorphism group for all finite dimensional real $C^*$-algebras.
We need a preliminary lemma.

\begin{lemma}\label{lemma:Iideal}
Let $ \A $ be a finite dimensional real $C^*$-algebra, $Q$ a CQG and $\alpha:\A_{\C}\to\A_{\C}\otimes Q$ a coaction,
and $I$ the $C^*$-ideal of $Q$ generated by elements
$$
( \varphi \otimes \id)\big(\alpha\circ\sigma ( a )  - ( \sigma \otimes \ast )\circ\alpha ( a )\big) \;,
$$
with $a\in\A_{\C}$ and $\varphi\in(\A_{\C})^*$. Then $I$ is a Woronowicz $C^*$-ideal of $Q$.
\end{lemma}

\begin{proof}
Let $a\in\A_{\C}$ and
$$
T:=\alpha\circ\sigma ( a )  - ( \sigma \otimes \ast )\circ\alpha ( a ) \;.
$$
The ideal $I$ is generated by
$$
t_\varphi= (\varphi\otimes\id)(T) \;,
\qquad
\varphi\in (\A_{\C})^*\;.
$$
Let $\pi_I:Q\to Q/I$ be the quotient map.
We need to prove that
$$
(\pi_I\otimes\pi_I)\Delta(t_\varphi)=
(\varphi\otimes\pi_I\otimes\pi_I)(\id\otimes\Delta)(T)
$$
is zero for all $\varphi$. It is enough to prove that
$(\id\otimes\pi_I\otimes\pi_I)(\id\otimes\Delta)(T)=0$.
We have
\begin{align*}
(\id\otimes\Delta)(T)
&=(\id\otimes\Delta)\circ\alpha\circ\sigma(a)-(\sigma\otimes\Delta\circ*)\circ\alpha(a) \\
&=(\id\otimes\Delta)\circ\alpha\circ\sigma(a)-(\sigma\otimes*\otimes*)\circ(\id\otimes\Delta)\circ\alpha(a) \\
&=(\alpha\otimes \id)\circ\alpha\circ\sigma(a)-(\sigma\otimes*\otimes*)\circ(\alpha\otimes \id)\circ\alpha(a) \;,
\end{align*}
where we used the fact that $\Delta$ is a $*$-homomorphism (and $\alpha$ a coaction).
Therefore, since
$$
(\id\otimes\pi_I)\circ\alpha\circ\sigma=
(\id\otimes\pi_I)\circ(\sigma\otimes*)\circ\alpha \;,
$$
we get
\begin{align*}
(\id\otimes\pi_I\otimes\pi_I)(\id\otimes\Delta)(T) \hspace{-3cm} \\
&=(\id\otimes\pi_I\otimes\id)(\alpha\otimes \id)\circ(\id\otimes\pi_I)\circ\alpha\circ\sigma(a)
-(\id\otimes\pi_I\otimes\pi_I)\big((\sigma\otimes*)\alpha\otimes*\big)\circ\alpha(a) \\
&=(\id\otimes\pi_I\otimes\id)(\alpha\otimes \id)\circ(\id\otimes\pi_I)\circ(\sigma\otimes*)\circ\alpha(a)
-(\id\otimes\pi_I\otimes\pi_I)(\alpha\circ\sigma\otimes*)\circ\alpha(a) \\
&=(\id\otimes\pi_I\otimes\pi_I)(\alpha\circ\sigma\otimes*)\circ\alpha(a)
-(\id\otimes\pi_I\otimes\pi_I)(\alpha\circ\sigma\otimes*)\circ\alpha(a)
\\
&=0 \;.
\end{align*}
This concludes the proof.
\end{proof}

If $\A$ is a finite-dimensional complex $C^*$-algebra, we denote by $\mathbf{C_{aut}}(\A)$ the category of CQGs with a coaction on $\A$ preserving the trace. It was proved in \cite{Wan98a} that this category has a universal object, here denoted by $A_{aut}(\A)$.

\begin{prop}\label{prop:QautR}
For a finite dimensional real $C^* $ algebra $ \A, $
the universal object in the category $\qaut(\A)$ exists and it is given by $A_{aut,\R}(\A)=A_{aut}(\A)/I$,
where $I$ is the ideal of $Q=A_{aut}(\A)$ defined in Lemma \ref{lemma:Iideal}.
\end{prop}

\begin{proof}
By Lemma \ref{lemma:Iideal}, $A_{aut}(\A)/I$ is a CQG, and by construction it satisfies \eqref{eq:alphasigma}.
Hence it is an object in $\qaut(\A)$.
Any object $Q$ of $\qaut(\A)$ is an object in $\mathbf{C_{aut}}(\A)$, hence
there exists a unique $C^*$-homomorphism $\phi:A_{aut}(\A)\to Q$ intertwining the coactions.
Since $Q$ satisfies \eqref{eq:alphasigma}, $\phi$ has $I$ in the kernel and then factorizes through
a map $A_{aut}(\A)\to A_{aut}(\A)/I\to Q$, proving that $A_{aut}(\A)/I$ is universal.
\end{proof}

\begin{rem}
The simplest example is $\A=\R^n$. In this case $\A_\C=\C^n$ is generated by $n$ orthogonal
projections $\delta_i$ with sum $1$, and $\sigma$ is the complex conjugation. In particular $\sigma(\delta_i)=\delta_i$.
The CQG $A_{aut}(\C^n)=A_s(n)$ is the quantum permutation group of \cite{Wan98a} (we use the notations of \cite{Ban05a}):
it is generated by a ``magic unitary'' $u=(u_{ij})$ i.e.~a matrix whose entries are projections $u_{ij}=u_{ij}^*=u_{ij}^2$
and on each row and column of $u$ these projections sum up to $1$. The coaction on $\C^n$ is $\delta_i\mapsto\sum_{j=1}^n
\delta_j\otimes u_{ij}$, and \eqref{eq:alphasigma} is trivially satisfied. Hence, $\qaut(\R^n)=A_s(n)$.
\end{rem}

\subsection{Quantum automorphisms of $M_n(\R)$ and $M_n(\Q)$}\label{sec:5.2}

We will need the following lemma for the purpose of the computation of the quantum automorphism group.

\begin{lemma}\label{francescoold1}
Let $Q$ be a $C^*$-algebra, which is generated by elements $a^{kl}_{ij}$ satisfying the equations (4.1-4.3) of \cite{Wan98a} such that $b^{ kl}_{ij}=(a^{kl}_{ij})^*$ also satisfy (4.1-4.3) of \cite{Wan98a}. Then $Q$ is commutative. 
\end{lemma}

\begin{proof}
By (4.2) of \cite{Wan98a}, we have $\sum_r a^{ij}_{kr}a^{mn}_{rl}=\delta_{jm}a^{in}_{kl},$ while the equation (4.1) of \cite{Wan98a} for $b^{ kl}_{ij}=a^{lk}_{ji}$ gives $\sum_m a^{mn}_{rl}a^{pm}_{qs}=\delta_{rs}a^{pn}_{ql}$.
Using these two equations, we get 
\begin{align*}
\sum_m(\sum_r a^{ij}_{kr}a^{mn}_{rl})a^{pm}_{qs}  &=\sum_m\delta_{jm} a^{in}_{kl}a^{pm}_{qs}=a^{in}_{kl}a^{pj}_{qs} \;,\\
\sum_r a^{ij}_{kr} (\sum_m a^{mn}_{rl}a^{pm}_{qs})&=\sum_r a^{ij}_{kr}\delta_{rs} a^{pn}_{ql}=a^{ij}_{ks}a^{pn}_{ql} \;.
\end{align*}
Hence $a^{kl}_{ij}a^{rs}_{mn}=a^{ks}_{in}a^{rl}_{mj}.$ A similar computation, but exchanging the roles of $a^{kl}_{ij}$ and $b^{ kl}_{ij}=(a^{kl}_{ij})^*=a^{lk}_{ji}$, gives $a^{rs}_{mn}a^{kl}_{ij}=a^{ks}_{in}a^{rl}_{mj}.$ These two equations together imply that the generators of $Q$ commute: $a^{rs}_{mn}a^{kl}_{ij}=a^{kl}_{ij}a^{rs}_{mn}$. This proves that the $C^*$-algebra $Q$ is commutative.
\end{proof}

\begin{prop}\label{prop:QautRofMnR}
$A_{aut,\R}(M_n(\R))$ is isomorphic to $ C ( PO(n) ).$
\end{prop}

\begin{proof}
We use Prop.~\ref{prop:QautR} to determine $A_{aut,\R}(M_n(\R))$.
Let $\{a^{kl}_{ij}\}_{i,j,k,l=1}^n$ be the generators of the quantum automorphism group $A_{aut}(M_n(\C))$, in the notation of Theorem 4.1 of \cite{Wan98a},
with coaction $\alpha(e_{ij})=\sum_{kl}e_{kl}\otimes a^{kl}_{ij}$ in the canonical basis $e_{ij}$ of $M_n(\C)$.
The additional condition \eqref{eq:alphasigma} gives:
$$
a^{kl}_{ij}=(a^{kl}_{ij})^* \;.
$$
Thus, both $a^{kl}_{ij}$ and $b^{kl}_{ij}=(a^{kl}_{ij})^*$ also satisfy equations (4.1-4.5) of \cite{Wan98a} and so by Lemma \ref{francescoold1}, $Q$ is commutative.
Thus $A_{aut,\R}(M_n(\R))=C(G)$, where $G$ is the automorphism group of $M_n(\R)$, i.e.~$G=PO(n)$.
\end{proof}

\begin{prop}
$A_{aut,\R}(M_n(\Q))$ is isomorphic with  $C(PSp(n)),$ where $Sp(n)$ is the quaternionic unitary group.
\end{prop}

\begin{proof}
Similarly to Prop.~\ref{prop:QautRofMnR},
let $\{a^{kl}_{ij}\}_{i,j,k,l=1}^{2n}$ be the generators of the quantum automorphism group $A_{aut}(M_{2n}(\C))$, in the notation of Theorem 4.1 of \cite{Wan98a},
with coaction $\alpha(e_{ij})=\sum_{kl}e_{kl}\otimes a^{kl}_{ij}$ in the canonical basis $e_{ij}$ of $M_{2n}(\C)$.
$A_{aut,\R}(M_n(\Q))$ is the quotient of $A_{aut}(M_{2n}(\C))$ by the relation \eqref{eq:alphasigma}, where now $\sigma$ is given by
Prop.~\ref{sigmaforrealalgebrasformulae}, point 3. This gives the relations
$$
a^{\tilde{k}\tilde{l}}_{\tilde{r}\tilde{s}}=(-1)^{r+s+k+l}a^{lk}_{sr} \;,
$$
where for $m\in\N$ we set $\tilde{m}:=m+(-1)^{m+1}$. We claim that $ A_{aut,\R}(M_n(\Q)) $ is commutative $C^*$-algebra. 
The proof of this claim is similar to the proof of Lemma \ref{francescoold1}. By (4.2) of \cite{Wan98a},
we have $\sum_ra_{kr}^{ij}a_{rl}^{mp}=\delta_{jm}a_{kl}^{ip}$,
while substituting $a^{lk}_{sr}=(-1)^{r+s+k+l}a^{\tilde{k}\tilde{l}}_{\tilde{r}\tilde{s}}$ in (4.1) we get
$$
\sum\nolimits_m
(-1)^{2m+k+l+p+q+r+s}a^{\tilde m\tilde k}_{\tilde r\tilde l}a^{\tilde p\tilde m}_{\tilde q\tilde s}=
(-1)^{k+l+p+q}\delta_{\tilde r\tilde s}a^{\tilde p\tilde k}_{\tilde q\tilde l} \;.
$$
Since the map $m\mapsto\tilde m$ is a bijection ($m$ runs from $1$ to $2n$),
summing over $m$ or $\tilde m$ makes no difference. Furthermore since $(-1)^{2m}=1$,
renaming all the labels we get the relation $(-1)^{r+s}\sum_ma^{mk}_{rl}a^{pm}_{qs}=\delta_{rs}a^{pk}_{ql}$.
Now we multiply both sides by $(-1)^{r+s}$ and get 
$\sum_ma^{mk}_{rl}a^{pm}_{qs}=(-1)^{r+s}\delta_{rs}a^{pk}_{ql}=\delta_{rs}a^{pk}_{ql}$
exactly as in the proof of Lemma \ref{francescoold1}. 
Repeating verbatim the proof of Lemma \ref{francescoold1} one can conclude that $ A_{aut,\R}(M_n(\Q)) $ is commutative, hence isomorphic to $C(G)$ where $G=PSp(n)$ is the well known classical group of automorphism of $M_n(\Q)$.
\end{proof}

\subsection{Quantum symmetries of $M_n(\C)$}\label{sec:5.3}

We need a preliminary lemma.

\begin{lemma}\label{lemma:5.9}
Let $\gamma$ denote the non-trivial generator of $\Z_2$, with action of $\gamma$
on $PU(n)$ induced by the map $u\mapsto\bar u$ on $U(n)$.
Then, the group of real automorphisms of $M_n(\C)$
is the semidirect product $PU(n)\rtimes\Z_2$.
\end{lemma}

\begin{proof}
In this proof, we will identify a scalar $\lambda\in\C$ with the matrix
$\lambda\mathbb{I}_n$ in $M_n(\C)$.

Let $\varphi:M_n(\C)\to M_n(\C)$ be an automorphism of real algebras. Then $\varphi(ia)=\varphi(i)\varphi(a)$ for any $a\in M_n(\C)$.
$\varphi(i)$ must be central, hence proportional to the identity, and satisfy $\varphi(i)^2=-1$.
Thus we have only two cases: $\varphi(i)=i$ (and $\varphi$ is complex linear) or $\varphi(i)=-i$
(and $\varphi$ is antilinear). Any antilinear automorphism is of the form $a\mapsto \varphi(Ja)$,
where $Ja=\bar a$ is a canonical antilinear automorphism and $\varphi$ is complex linear.
Any complex linear automorphism $\varphi$ is inner, hence of the form $\varphi_u(a)=uau^*$ with $u\in U(n)$,
and $J\varphi_u J=\varphi_{\bar u}$. This induces an action of $\Z_2$ on $PU(n)$. The classical
group of automorphisms is then $PU(n)\rtimes\Z_2$.
\end{proof}

\begin{prop}
$A_{aut,\R}(M_n(\C))$ is isomorphic to $C(PU(n)\rtimes\Z_2)$.
\end{prop}

\begin{proof}
Recall that for $\A=M_n(\C)$, $\A_{\C}=M_n(\C)\oplus M_n(\C)$
and $\sigma(a,b)=(\bar b,\bar a)$. 

Let $\{a^{kl}_{ij,xy}\}_{i,j,k,l=1}^n,~x,y=1,2$ be the generators of the quantum automorphism group $Q=\A_{aut}(M_n(\C)\oplus M_n(\C))$, in the notation of Theorem 5.1 of \cite{Wan98a}.
Notice that for any fixed $r,s$, from equations (5.1)-(5.3) we see
that $a^{kl}_{ij,rs}$ satisfy (4.1)-(4.3).

It follows from Prop.~\ref{prop:QautR} that $A_{aut,\R}(M_n(\C))$ is the quotient of  
$Q$ by the relation (coming from \eqref{eq:alphasigma}):
$$
a^{kl}_{ij,11}=(a^{kl}_{ij,22})^* \;,\qquad
a^{kl}_{ij,12}=(a^{kl}_{ij,21})^* \;.
$$
Since both $a^{kl}_{ij,11}$ and $a^{kl}_{ij,22}=(a^{kl}_{ij,11})^*$
satisfy the equations (4.1-4.3) of \cite{Wan98a}, by Lemma \ref{francescoold1},
they generate a commutative $C^*$-subalgebra $Q_1\subset Q$.
Similarly $a^{kl}_{ij,12}$ generate a commutative $C^*$-subalgebra
$Q_2\subset Q$.

Note that 
$a^{kl}_{ij,22}=(a^{kl}_{ij,11})^*=a^{lk}_{ji,11}$ and
$a^{kl}_{ij,21}=(a^{kl}_{ij,12})^*=a^{lk}_{ji,12}$, thus
$a^{kl}_{ij,11}$ and $a^{kl}_{ij,12}$ are a complete set
of generators, and $A_{aut,\R}(M_n(\C))$ is a quotient of the free product $Q_1*Q_2$.
With $a^{kl}_{ij}:=a^{kl}_{ij,11}$ and $b^{kl}_{ij}:=a^{kl}_{ij,12}$,
from (5.1) of \cite{Wan98a} we get
\begin{align*}
\sum\nolimits_ma^{km}_{ij,11}a^{ml}_{pq,12}=
\sum\nolimits_ma^{km}_{ij}b^{ml}_{pq} &=0 \;,&
\left(\sum\nolimits_ma^{km}_{pq,21}a^{ml}_{ij,22}\right)^*=
\sum\nolimits_ma^{ml}_{ij}b^{km}_{pq} &=0 \;,\\
\sum\nolimits_ma^{km}_{ij,12}a^{ml}_{pq,11}=
\sum\nolimits_mb^{km}_{ij}a^{ml}_{pq} &=0 \;,&
\left(\sum\nolimits_ma^{km}_{pq,22}a^{ml}_{ij,21}\right)^*=
\sum\nolimits_mb^{ml}_{ij}a^{km}_{pq} &=0  \;.
\end{align*}
plus the similar ones where one sums over lower indices.
Now, with a trick similar to the proof of Lemma \ref{francescoold1},
we compute
\begin{align*}
\sum_m(\sum_r a^{ij}_{kr}a^{mn}_{rl})b^{pm}_{qs}  &=\sum_m\delta_{jm} a^{in}_{kl}b^{pm}_{qs}=a^{in}_{kl}b^{pj}_{qs} \;,\\
\sum_r a^{ij}_{kr} (\sum_m a^{mn}_{rl}b^{pm}_{qs})&=\sum_r a^{ij}_{kr}\cdot 0=0 \;.
\end{align*}
This proves that $a^{kl}_{ij}b^{rs}_{pq}=0$ for all the values of
the labels. Repeating the same, but exchanging the role of $a$ and
$b$, we get $b^{rs}_{pq}a^{kl}_{ij}=0$ too. Hence, as a $C^*$-algebra
$Q=Q_1\oplus Q_2$ is commutative.
The CQG isomorphism $Q\simeq C(PU(n)\rtimes\Z_2)$ follows from Lemma \ref{lemma:5.9}.
\end{proof}

We end this article by identifying some categories of ``quantum symmetries'' of $M_n(\C)$,
whose universal objects are the half liberated quantum unitary and orthogonal
groups and the free quantum orthogonal group.

\begin{prop}
Consider the category of pairs $(Q,U)$, where $Q$ has a unitary corepresentation $U$ on $\C^n$ so that the adjoint action extends to a quantum automorphism of the real $C^*$-algebra $M_n(\C).$ Then the universal object in this category exists and is isomorphic to $A_u^*(n).$ 

Consider the subcategory consisting of pairs $(Q,U),$ such that $U\circ J=(J\ot \ast)\circ U,$ where $J$ is the componetwise conjugation on $\C^n.$ Then, the universal object in this category exists and is isomorphic to $A_o^*(n).$
\end{prop}

\noindent
We omit the proof of this proposition since it is very similar to the proof of Prop.~5.3 of \cite{ludwik}.

We recall now briefly the notion of quantum isometry group from \cite{Gos10}.
This generalizes the notion of orientation preserving isometries of a closed
Riemannian spin manifold to the framework of spectral triples and CQGs.

For a finite-dimensional odd spectral triple $(\A,\HH,0,J)$, with Dirac operator $D=0$,
the definition is as follows \cite{Gos10}.
As usual, we choose an orthonormal basis for $\HH$, and denote by $J_0$ the composition
of $J$ with the complex conjugation on the components of $\HH$.

\begin{df}\label{def:qisot}
A pair $(Q,U)$ is a quantum family of
``orientation and real structure preserving isometries'' for
$(\A,\HH,0,J)$ if
$Q$ is a unital $C^*$-algebra and $U$ is a unitary element of $\B(\HH)\otimes Q$
such that
$$
(J_0\otimes 1_Q)\bar U=U(J_0\otimes 1) \;,\qquad
\mathrm{Ad}_U(\A)\subset\A\otimes Q \;.
$$
\end{df}

\noindent
The category with objects $ ( Q, U ) $ as in Def.~\ref{def:qisot} and morphisms given by unital $*$-ho\-mo\-mor\-phisms intertwining the corepresentations,
has a universal object denoted by $\qisot(\A,\HH,0,J)$.
It has a structure of a CQG and the associated unitary operator, say $U_0$, is a faithful unitary corepresentation. The quantum isometry group $\qiso(\A,\HH,0,J)$ is the Woronowicz $C^*$-subalgebra of $\qisot(\A,\HH,0,J)$ generated by the elements $ \{ ( \varphi \otimes \mathrm{id} ) \mathrm{Ad_{U_0}} ( a ) : a \in \A, ~ \varphi \in \A^* \} $.

\begin{prop}
Let $J$ be the antilinear map on $\C^n$ given by complex conjugation on the components. Then $\qisot(M_n(\C),\C^n,0,J)=A_o(n)$ and $\qiso(M_n(\C),\C^n,0,J)=PA_o(n)$. %
\end{prop}
\begin{proof}
Let $\{e_i\}_{i=1}^n$ be a basis of $\C^n$ and $U(e_i)=\sum_{j}e_j\otimes a_{ji},$ where $a_{ji}\in Q.$ Then the relation $UJ(e_i)=\{(J\otimes\ast)U\}(e_i)$ implies that $a_{ij}=a^\ast_{ij}$ for all $i,j=1,2,\ldots,n.$ Thus the matrix $u:=(a_{ji})$ is a unitary and satisfies $u=\overline{u}$ implying that $Q$ is a quantum subgroup of $A_o(n).$ Since $A_o(n)$ belongs to the category, this concludes the proof.
\end{proof}

\smallskip

\begin{center}
\textsc{Acknowledgments}
\end{center}
\vspace{-4pt}
We would like to thank T.~Banica for useful discussions and J.W.~Barrett for motivating us to work on quantum gauge groups. 
We thank M.~Dubois-Violette for pointing out the papers \cite{MDV90} and \cite{Bic03a} and for several interesting comments,
and S.~Wang for careful reading the first version of the paper and for several clarifications and suggestions for improvements.

A part of this work was done when J.B.~was visiting IHES. F.D.~was partially supported by the italian ``Progetto FARO 2010'' and by the ``Progetto Giovani GNFM 2011'' (INDAM, Italy). B.D.~would like to thank J.M.~Lindsay for his continuous support and encouragement, during the visit to Lancaster University. He also acknowledges the support of UKIERI project entitled ``Quantum Probability, Noncommutative geometry and Quantum Information'' (QP, NCG \& QI). L.D.~was partially supported by GSQS 230836 (IRSES, EU) and PRIN 2008 (MIUR, Italy).

\noindent{\small
\textit{Email addresses:} \texttt{jyotishmanb@gmail.com},
\texttt{francesco.dandrea@unina.it}, \texttt{dabrow@sissa.it},\linebreak
\texttt{biswarupnow@gmail.com}.}

\end{document}